\documentclass[twoside,11pt]{article}

\usepackage{blindtext}

\usepackage[plainnat, preprint, sort&compress]{jmlr2e}
\usepackage{amsmath}
\usepackage{tikz}
\usetikzlibrary{patterns}
\usepackage{todonotes}
\usepackage{bigstrut,multirow}
\usepackage{setspace}
\usepackage{subfigure}
\usepackage{enumitem}
\usepackage{algorithm}
\usepackage{algpseudocode}  

\usepackage{amssymb}        
\usepackage{caption}        
\usepackage{float}          

\let\cite\citep

\newtheorem{assumption}{Assumption}
\newenvironment{acknowledgements}
{\section*{Acknowledgements}}
{}

\usepackage{lastpage}

\ShortHeadings{Proximal methods for nonsmooth manifold optimization}{Li et al.}
\firstpageno{1}

\begin{document}

\title{Proximal methods for structured nonsmooth optimization over Riemannian submanifolds}

\author{\name Qia Li\thanks{These authors contributed equally to this work.} \email liqia@mail.sysu.edu.cn \\
       \addr School of Computer Science and Engineering\\
       Guangdong Province Key Laboratory of Computational Science\\
       Sun Yat-sen University\\
       Guangzhou, China
       \AND
       \name Na Zhang\footnotemark[\value{footnote}]        \thanks{Corresponding author.} \email nzhsysu@gmail.com  \\
       \addr Department of Applied Mathematics\\
       College of Mathematics and Informatics\\
       South China Agricultural University\\
       Guangzhou, China
       \AND
       \name Junyu Feng \email fengjy63@mail2.sysu.edu.cn\\
       \addr School of Computer Science and Engineering\\
       Sun Yat-sen University\\
       Guangzhou, China
       \AND
       \name Hanwei Yan \email yanhw5@mail2.sysu.edu.cn\\
       \addr School of Computer Science and Engineering\\
       Sun Yat-sen University\\
       Guangzhou, China}

\editor{My editor}

\maketitle

\begin{abstract}
In this paper, we consider a class of structured nonsmooth optimization problems over an embedded submanifold of a Euclidean space, where the first part of the objective is the sum of a difference-of-convex (DC) function and a smooth function, while the remaining part is a weakly convex function over a smooth function. This model problem has many important applications in machine learning and scientific computing, for example, the sparse Fisher discriminant analysis. We propose a manifold proximal-gradient-subgradient algorithm (MPGSA) and show that under mild conditions any accumulation point of the solution sequence generated by it is a critical point of the underlying problem. By assuming the Kurdyka-{\L}ojasiewicz property of an auxiliary function, we further establish the convergence of the full sequence generated by MPGSA under some suitable conditions. When the second component of the DC function involved is the maximum of finite continuously differentiable convex functions, we also propose an enhanced MPGSA with guaranteed subsequential convergence to a lifted B-stationary points of the optimization problem. Finally, some preliminary numerical experiments are conducted to illustrate the efficiency of the proposed algorithms.
\end{abstract}


\section{Introduction}

\label{sec:intro}

	In modern machine learning and data science, many optimization models naturally involve manifold constraints and nonsmooth structures arising from sparsity. Nevertheless, these models often lead to highly nonconvex and nonsmooth formulations due to the presence of regularization terms and nonlinear geometric constraints. Consequently, developing efficient algorithms with solid convergence guarantees for such structured problems on manifolds has become a central topic in the intersection of machine learning and optimization. Motivated by these challenges, in this paper we consider the following structured nonsmooth and nonconvex minimization problem:
	\begin{equation} \label{problemn1}
		\min \limits_{x \in \mathcal{M}} \left\{ F(x) := h(x) - \frac{f(x)}{g(x)} + r(x) \right\},
	\end{equation}  
	where $\mathcal{M}$ is an embedded submanifold in $\mathbb{R}^n$, $h:\mathbb{R}^n \to \mathbb{R}$ is a difference-of-convex (DC) function, $r:\mathbb{R}^n\to\mathbb{R}$ and $g:\mathbb{R}^n\to (0,+\infty)$ are smooth, and $f:\mathbb{R}^n\to \mathbb{R}$ is possibly nonsmooth. More precisely, we make the following assumptions for problem \eqref{problemn1} throughout this paper.
	
	\begin{assumption} \label{ass1} 
		\indent
		\begin{enumerate}[label = {\upshape(\roman*)}]
			\item $\mathcal{M}$ is a smooth embedded submanifold in $\mathbb{R}^n$ of dimension $n$-$d$ with a nonnegative integer $d$.
			\item $h:= h_1-h_2$, where $h_1:\mathbb{R}^n \to \mathbb{R}$, $h_2:\mathbb{R}^n \to \mathbb{R}$ are both convex (possibly nonsmooth), and $h_1$ is Lipschitz continuous with the Lipschitz constant $L_1$.
			\item $g$ and $r$ are continuously differentiable (possibly nonconvex) with locally Lipschitz continuous gradients $\nabla g$ and $\nabla r$, respectively.
			\item $f$ is weakly convex with the modulus $\tau > 0$ on $\mathbb{R}^n$ and $g$ takes positive values on $\mathcal{M}$.
		\end{enumerate}
	\end{assumption}
	
	Note that here the smoothness, Lipschitz continuity and (weak) convexity of functions are interpreted in the ambient Euclidean space. The proposed model \eqref{problemn1} naturally arises in many machine learning problems involving manifold constraints and nonsmooth regularization. The following examples illustrate typical instances of model \eqref{problemn1}.\\

	\textbf{Example 1. Sparse Fisher discriminant analysis.} 
	Fisher discriminant analysis (FDA) \cite{Fisher1936} is an extensively used dimension reduction method for supervised classification. It seeks for optimal linear combination of features which maximizes the between-class variance over the within-class variance, thereby facilitating effective classification. Given an $n$-dimensional multi-class dataset, let $A_b \in \mathbb{R}^{n \times n}$ and $A_w \in \mathbb{R}^{n \times n}$ be the associated between-class covariance and within-class covariance matrix, respectively. It is well-known that $A_b$ is symmetric positive semidefinite, while $A_w$ is symmetric positive definite. Then, the FDA can be formulated into 
	\begin{equation} \label{model:FDA}
		\min \limits_{X\in \mathcal{S}_{n,p}} \left\{- \frac{\mathrm{tr}(X^T A_b X)}{\mathrm{tr}(X^T A_w X)} \right\},
	\end{equation}
	where $\mathrm{tr}(\cdot)$ denotes the trace of a square matrix and $\mathcal{S}_{n,p}:=\{X\in\mathbb{R}^{n\times p}:X^TX=I_p\}$ is the Stiefel manifold. However, the FDA may suffer from overfitting and computational inefficiency in high dimensional data scenarios. To address these challenges, the sparse FDA \cite{Clemmensen-Hastie-Witten-Ersboll:Technometrics:2011} introduces sparsity into the original FDA, which yields the following optimization problem
	\begin{equation} \label{model:SFDA}
		\min \limits_{X\in \mathcal{S}_{n,p}} \left\{ \lambda \varphi(X) - \frac{\mathrm{tr}(X^T A_b X)}{\mathrm{tr}(X^T A_w X)} \right\},
	\end{equation}
	where $\lambda > 0$ is a weighting parameter and $\varphi:\mathbb{R}^{n\times p}\to\mathbb{R}$ is a sparse regularization function. We note that many widely used $\varphi$ are actually DC functions, where the first component is the $\ell_1$ norm, e.g., the SCAD \cite{Fan-Li:JASA:2002}, MCP \cite{Zhang:TAS:2010}, capped $\ell_1$ function \cite{Zhang:JMLR:2010} and partial $\ell_1$ function \cite{Gotoh-Takeda-Tono:MP:2018}. With these choices of $\varphi$, problem \eqref{model:SFDA} is a special case of model \eqref{problemn1} with $\mathcal{M}=\mathcal{S}_{n,p}$, $h(X)=\lambda\varphi(X)$, $f(X)=\mathrm{tr}(X^TA_b X)$, and $g(X)=\mathrm{tr}(X^T A_w X)$. Moreover, direct verifications can show the Lipschitz continuity of the $\ell_1$ norm as well as the convexity of $\mathrm{tr}(X^TA_b X)$, and thus Assumption~\ref{ass1} is fulfilled for the sparse FDA problem \eqref{model:SFDA}. 
	
	\textbf{Example 2. Sparse sharpe ratio optimization.} Suppose that an investing strategy is represented by a portfolio $w\in\mathbb{R}^n$ of $n$ assets, $r>0$ denotes the risk-free return, $b\in[0,+\infty)^n$ and $Q=(q_{ij})_{n\times n}\in\mathbb{R}^{n\times n}$ stand for the expected return rate and its covariance matrix respectively. The classical sharpe ratio optimization problem can be formulated into
	\begin{equation}\label{ex2:1}
		\min_{w\in\mathcal{F}} \left\{-\frac{b^Tw-r}{\sqrt{w^TQw}}\right\},
	\end{equation}
	where $\mathcal{F}:=\left\{w\in\mathbb{R}^n:\sum\limits_{i=1}^n w_i=1, w_i\geq 0, i=1,2,\dots,n\right\}$ denotes the simplex set. Note that the numerator $b^Tw-r$ can take both nonnegative and negative values on the feasible set $\mathcal{F}$. In modern portfolio selection, sparse portfolios are more desirable for the sake of saving manage and financial costs. To obtain sparse portfolios, one may introduce the $\ell_{1/2}$-regularization into \eqref{ex2:1}, and resort to the following model:
	\begin{equation}\label{ex2:2}
		\min_{w\in\mathcal{F}} \left\{\lambda\|w\|^{1/2}_{1/2}-\frac{b^Tw-r}{\sqrt{w^TQw}}\right\},
	\end{equation}
	where $\|w\|_{1/2}=\left(\sum\limits_{i=1}^{n}|w_i|^{1/2}\right)^2$. Next we show the model \eqref{ex2:2} can be rewritten as a special case of \eqref{problemn1}. To this end, we introduce a variable $x\in\mathbb{R}^n$ such that $x_i^2=w_i$, for $i=1,2,\dots,n$. Replacing $w$ by $x$ in \eqref{ex2:2}, we obtain the following equivalent optimization problem 
	\begin{equation}\label{ex2:3}
		\min_{x\in \mathcal{S}_{n,1}} \left\{\lambda\| x\|_1-\frac{\sum\limits_{i=1}^n b_ix_i^2-r}{\sqrt{\sum\limits_{i=1}^n\sum\limits_{j=1}^n q_{ij}x_i^2x_j^2}}\right\}.
	\end{equation}
	It is easy to verify that \eqref{ex2:3} is a special instance of \eqref{problemn1} and satisfies Assumption~\ref{ass1}.
	
	\textbf{Example 3. Sparse $\ell_1$ norm principal component analysis.} 
	Sparse principal component analysis (PCA) \cite{Zou-Hastie-Tibshirani:JCGS:2006, Zou-Hastie-Tibshirani:PIEEE:2018} aims to enhance the interpretability of the classical PCA method by seeking for principle components of the given data with very few nonzero elements. Concretely, let $B \in \mathbb{R}^{n \times m}$ be the data matrix, where $m$ and $n$ denote the number of samples and dimension of the features, respectively. Assume that the sample mean of $B$ is zero. Then, the sparse PCA is commonly formulated as 
	\begin{equation} \label{model:SPCA}
		\min \limits_{X\in \mathcal{S}_{n,p}} \left\{- \mathrm{tr}(X^T BB^T X) + \mu \| X \|_1 \right\},
	\end{equation}
	where $\mu > 0$ is a weighting parameter and $\| \cdot \|_1$ denotes the $\ell_1$ norm. Although the sparse PCA has achieved great success in many applications, it is prone to be adversely affected by outliers. To alleviate this issue, the sparse $\ell_1$-norm PCA has been proposed \cite{Meng-Zhao-Xu:PR:2012}, which can be written as
	\begin{equation} \label{model:L1SPCA}
		\min \limits_{X\in \mathcal{S}_{n,p}} \left\{- \| B^T X \|_1 + \mu \| X \|_1 \right\}.
	\end{equation}
	Obviously, by specifying $\mathcal{M}=\mathcal{S}_{n,p}, h_1(X) = \mu \| X \|_1, h_2(X) = \| B^T X \|_1, \text{and}\ f = r = 0$, model \eqref{problemn1} reduces to sparse $\ell_1$-norm PCA problem \eqref{model:L1SPCA}, which also satifies Assumption~\ref{ass1}.

	The nonsmooth manifold optimization has drawn a lot of attentions in the recent years. Various methods have been designed for nonsmooth Riemannian manifold optimization, such as Riemannian gradient sampling algorithms \cite{Grohs-Hosseini:ACM:2016,Hosseini-Huang-Yousefpour:SIOPT:2018,Hosseini-Uschmajew:SIOPT:2017}, Riemannian subgradient-type algorithms     \cite{Borckmans-Selvan-Boumal-Absil:JCAM:2014,Ferreira-Louzeiro-Prudente:Opt:2019,Ferreira-Oliveira:OTA:1998,Li-Chen-Deng-Qu-Zhu-So:SIOPT:2021}, Riemannian proximal point methods \cite{De-Glaydston-Da-Oliveira:OTA:2016,Chen-Ma-Anthony-Zhang:SIOPT:2020,Huang-Wei:MP:2022,Si-Absil-Huang-Jiang-Vary:SIOPT:2024, Wang-Liu-Chen-Ma-Xue-Zhao:IJOO:2021, Wang-Ma-Xue:JMLR:2022}, operator splitting methods \cite{Chen-Ji-You:JSC:2016,Lai-Osher:JSC:2014,Zhou-Bao-Ding-Zhu:MP:2023},  infeasible manifold optimization methods \cite{Hu-XIao-Liu-Toh:arXiv:2022,Liu-Xiao-Yuan:JSC:2024,Xiao-Liu-Yuan:SIOPT:2021} and so on. It is worth noting that many works focus on solving manifold optimization problems whose objective function is the sum of a nonsmooth convex function and a smooth function.~This kind of problems can be viewed as particular cases of problem \eqref{problemn1} with $h_2=f=0$ and $g \equiv 1$. However, due to the presence of the nonsmooth nonconvex terms $-h_2$ and $-f/g$, the aforementioned manifold optimization methods can not be directly applied to solving problem \eqref{problemn1} in general cases. For instance, the direct application of ManPG algorithm \cite{Chen-Ma-Anthony-Zhang:SIOPT:2020} requires addressing a nonconvex and nonsmooth subproblem involving the two nonsmooth terms in each iteration, which is generally expensive and difficult to solve. Although ADMM-type methods can be facilitated to split the nonsmooth terms, the resulting algorithms usually have no convergence guarantee, to the best of our knowledge. Recently, proximal algorithms and Riemannian difference-of-convex algorithms (DCA) are developed in \cite{almeida2020modified,bergmann2024difference,souza2015proximal} for solving DC minimization problems over Hadamard manifolds. Unfortunately, these methods can not be applied to problem \eqref{problemn1} generally. On one hand, since the convexity considered in these works is not the convexity in the Euclidean space as in our setting but the geodesic convexity, their algorithms can not well tackle problem \eqref{problemn1} even when the nonsmooth ratio term $-f/g$ is absent. On the other hand, the underlying manifolds considered in these methods are restricted to be Hadamard manifolds rather than general Riemannian submanifolds as in problem \eqref{problemn1}. The above analysis forms the basic motivation of our work. We aim at developing efficient algorithms for solving problem \eqref{problemn1} by simultaneously exploiting the intrinsic structure of the objective and manifold. 
	
	Specifically, the contributions of this paper are summarized as follows. First, we introduce several notions of stationary points for nonsmooth DC and fractional optimization over a Riemannian submanifold and examine their relationships. Second, we propose a novel algorithm called manifold proximal-gradient-subgradient algorithm (MPGSA) for solving problem \eqref{problemn1} and establish its subsequential convergence towards a critical point of the problem under a mild assumption. To the best of our knowledge, the MPGSA is the first manifold optimization algorithm which can directly deal with problem \eqref{problemn1}. In particular, when $h_2=f=0$, $g \equiv 1$ and $\mathcal{M}$ is the Stiefel manifold, the MPGSA reduces to the ManPG method proposed in \cite{Chen-Ma-Anthony-Zhang:SIOPT:2020} and its subconvergence results follows immediately. Third, the convergence of the full sequence generated by MPGSA is established when the algorithm is equipped with a properly chosen stepsize, under the additional assumptions that $f$ is convex with a continuous conjugate function and a certain auxiliary function satisfies the Kurdyka–Łojasiewicz (KL) property. It is worth noting that the introduction of retraction operators poses inherent challenges in proving full sequential convergence for manifold proximal algorithms under KL assumptions, to the extent that even for the ManPG algorithm, full sequential convergence has not yet been established. In our convergence analysis, we overcome these challenges by ingeniously constructing an auxiliary function that incorporates the indicator function of the tangent bundle. As a beneficial byproduct, our convergence results also imply full sequential convergence of the ManPG algorithm. Finally, we develop an enhanced MPGSA (EMPGSA) that has subsequential convergence towards a lifted B-stationary point of problem \eqref{problemn1}, in the case where $h_2$ is the maximum of finite continuously differentiable convex functions.

	In addition to the aforementioned contributions, we would like to emphasize that our algorithmic developments and convergence analysis are not merely straightforward extensions of existing results in Euclidean fractional programming; they remain novel even when the underlying manifold reduces to the Euclidean space. Indeed, existing Euclidean fractional optimization methods (see, e.g. \cite{Boct-Dao-Li:SIOPT:2023,Shen-Yu:IEEE:TSP:2018}) handle the fractional objective in (1) by leveraging the square root of the numerator, i.e., $\sqrt{f}$, which requires the non-negativity of $f$ and the weak convexity of $\sqrt{f}$ over the feasible set. As a result, these methods cannot be applied to fractional programs where the numerator may take negative values on the feasible set, e.g., the aforementioned Example 2. In contrast, our algorithmic framework does not rely on $\sqrt{f}$, but instead builds upon the weak convexity of $f$ itself, thereby extending the applicability of our approach to many fractional programs arising in practical settings. Furthermore, under some KL assumption, we establish sequential convergence of the proposed MPGSA in cases where $f$ is either smooth or convex with a continuous conjugate. This result goes beyond existing Euclidean fractional optimization approaches, for which sequential convergence has previously been established only in the smooth case.
	
	The remaining parts of this paper are organized as follows. In section 2, we introduce notation and present some preliminaries. In section 3, we discuss different versions of stationarity for problem \eqref{problemn1}. In section 4 and 5, we propose MPGSA and EMPGSA, and establish their subsequential convergence. The sequential convergence of MPGSA is further analyzed in section 6. Numerical experiments are presented in section 7. Finally, we conclude this paper in section 8.
	
	\section{Notation and Preliminaries.}
	We begin with our preferred notations. We denote by $\mathbb{R}^n$ the $n$-dimensional Euclidean space. Here, $\mathbb{R}^n$ does not only refer to a vector space, but also can refer to a matrix space or a tensor space. The Euclidean inner product of $x\in\mathbb{R}^n$ and $y\in \mathbb{R}^n$ is denoted by $\langle x,y \rangle$, and the Frobenius norm of $x\in\mathbb{R}^n$ is denoted by $\|x\|=\sqrt{\langle x,x \rangle}$. For $x\in\mathbb{R}^n$, we use $\mathcal{B}(x)$ to denote a neighborhood of $x$, while use $\mathcal{B}(x,\delta)$ (resp., $\overline{\mathcal{B}}(x,\delta)$) to denote an open (resp., closed) ball centered at $x$ with radius $\delta>0$. The adjoint operator of a linear operator $\mathcal{A}:\mathbb{R}^n\to\mathbb{R}^d$ is denoted by $\mathcal{A}^{*}$. Let $\mathrm{dist}(x,S) := \inf \{\|x-y\|:y\in S\}$ be the distance from $x\in \mathbb{R}^n$ to the set $S \subseteq \mathbb{R}^n$. For a positive integer $d$, let $0_d$ and $I_d$ denote the $d$-dimensional zero vector and $d\times d$ identity matrix, respectively. For a function $\varphi:\mathbb{R}^n\rightarrow \mathbb{R}$, we use $\varphi\vert_\mathcal{M}$ to denote the restriction of $\varphi$ on $\mathcal{M}$.
	
	\subsection{Riemannian submanifold of $\mathbb{R}^n$.}
	We first review the standard definition of an embedded submanifold. Roughly speaking, an embedded submanifold in $\mathbb{R}^n$ is either an open subset or a smooth surface in the space \cite{Absil-Mahony-Sepulchre:Opt_Matrix_Man:2009,Boumal:Intro_Opt_Man:2023}. To this end, we first recall the differential operator of a vector-valued function. Let $\psi:\mathcal{O}\subseteq \mathbb{R}^n\rightarrow \mathbb{R}^m$ be defined on an open subset $\mathcal{O}$ of $\mathbb{R}^n$. We say that $\psi$ is smooth if it is infinitely differentiable on its domain. If $\psi$ is smooth, then for any $x\in\mathcal{O}$, the differential of $\psi$ at $x$ is the linear map $\mathrm{D} \psi(x):\mathbb{R}^n \rightarrow \mathbb{R}^m$ defined by $D\psi(x)(v)=\lim \limits_{t\rightarrow0}\frac{\psi(x + tv)-\psi(x)}{t}$, for all $v\in\mathbb{R}^n$.
	\begin{definition}\label{definition:embeddedSubmanifold}(embedded submanifolds of $\mathbb{R}^n$\cite{Boumal:Intro_Opt_Man:2023})
		Let $\mathcal{M}$ be a subset of a Euclidean space $\mathbb{R}^n$. We say $\mathcal{M}$ is a (smooth) embedded submanifold of $\mathbb{R}^n$ of dimension $m$ if either $m=n$ and $\mathcal{M}$ is an open subset of $\mathbb{R}^n$ or $m=n-d$ for a fixed integer $d\geq 1$ and for each $x\in \mathcal{M}$, there exists a neighborhood $\mathcal{B}(x)$ in $\mathbb{R}^n$ and a smooth function $\psi:\mathcal{B}(x)\to\mathbb{R}^d$ such that
		\begin{itemize}
			\item[(a)] if $y$ is in $\mathcal{B}(x)$, then $\psi(y)=0_d$ if and only if $y\in\mathcal{M}$, and
			\item[(b)] $\mathrm{rank}~\mathrm{D}\psi(x)=d$.
		\end{itemize}
		Such a function $\psi$ is called a local defining function for $\mathcal{M}$ at $x$.
	\end{definition}
	
	Let $\mathcal{M}$ be an embedded submanifold of $\mathbb{R}^n$ of dimension $m$. Then, for each $x \in \mathcal{M}$ there exist a neighborhood $U$ of $x$, an open subset $\hat{U} \subseteq \mathbb{R}^m$ and a map $\phi:\ U \rightarrow \hat{U}$ such that $\phi$ is a homeomorphism. The pair $(U, \phi)$ is called a chart. If $\mathcal{M}$ is an open submanifold, then the tangent space to $\mathcal{M}$ at $x\in\mathcal{M}$, denoted by $\mathrm{T}_x\mathcal{M}$, is $\mathbb{R}^n$. Otherwise, $\mathrm{T}_x\mathcal{M}=\mathrm{ker}~\mathrm{D}\psi(x)$ with $\psi$ any local defining function at $x$ \cite[Theorem 3.15]{Boumal:Intro_Opt_Man:2023}. The normal space to $\mathcal{M}$ at $x\in\mathcal{M}$, denoted by $\mathrm{N}_x\mathcal{M}$, is defined as the orthogonal complement space of $\mathrm{T}_x\mathcal{M}$ \cite{Si-Absil-Huang-Jiang-Vary:SIOPT:2024}. It should be noted that $\mathcal{M}$ is regular at $x\in\mathcal{M}$, and the tangent and normal cones to $\mathcal{M}$ at $x$ are exactly $\mathrm{T}_x\mathcal{M}$ and $\mathrm{N}_x\mathcal{M}$, respectively \cite[Example 6.8]{Rockafellar-Wets:VariationalAnalysis:2009}. Since the tangent space to $\mathcal{M}$ at $x$ is a vector space, one can equip it with an inner product or a metric. A manifold whose tangent spaces are endowed with a smoothly varying metric is referred to as a Riemannian manifold. If the manifold $\mathcal{M}$ is an embedded submanifold of $\mathbb{R}^n$ equipped with the Riemannian metric obtained by restriction of the metric of $\mathbb{R}^n$, then $\mathcal{M}$ is called a Riemannian submanifold of $\mathbb{R}^n$ \cite[Definition 3.55]{Boumal:Intro_Opt_Man:2023}. Throughout this paper, we assume the manifold $\mathcal{M}$ is a Riemannian submanifold of $\mathbb{R}^n$.
	
	The tangent bundle of a manifold $\mathcal{M}$ is the disjoint union of the tangent spaces to $\mathcal{M}$: $\mathrm{T}\mathcal{M}=\{(x,v):x\in\mathcal{M}\mbox{ and }v\in\mathrm{T}_x\mathcal{M}\}$ \cite[Definition 3.42]{Boumal:Intro_Opt_Man:2023}. If $\mathcal{M}$ is an embedded submanifold of $\mathbb{R}^n$ of dimension $n$-$d$, the tangent bundle $\mathrm{T}\mathcal{M}$ is an embedded submanifold of $\mathbb{R}^n\times\mathbb{R}^n$ of dimension $2(n-d)$ \cite[Theorem 3.43]{Boumal:Intro_Opt_Man:2023}. A retraction \cite[Definition 3.47]{Boumal:Intro_Opt_Man:2023} on a manifold is a smooth mapping $\mathrm{Retr}:\mathrm{T}\mathcal{M} \to \mathcal{M}$ such that:
	\begin{itemize}
		\item[(i)] $\mathrm{Retr}_x(0)=x$, where $0$ is the zero vector in $\mathrm{T}_x\mathcal{M}$;
		\item[(ii)] $\mathrm{D}\mathrm{Retr}_x(0)$ is the identity map,
	\end{itemize}
	where $\mathrm{Retr}_x:\mathrm{T}_x\mathcal{M} \rightarrow \mathcal{M}$ denotes the restriction of $\mathrm{Retr}$ to $\mathrm{T}_x\mathcal{M}$. The above definition also yields that 
	\begin{equation} \label{eq:Retr_definition_item2}
		\lim \limits_{0_n \neq v\in \mathrm{T}_x\mathcal{M} \atop 0_n\neq v\to 0_n} \frac{\|\mathrm{Retr}_x(v) - (x+v)\|}{\|v\|} = 0.
	\end{equation}
	%

	To end this subsection, we present some useful properties regarding the retraction.
	
	\begin{lemma} \label{lemma:Retr_property}
		Let $\mathcal{M}$ be an embedded submanifold of $\mathbb{R}^n$ and $\Lambda \subseteq \mathcal{M}$ be compact. Given $\delta>0$, the following statements hold.
		\begin{enumerate}[label = {\upshape(\roman*)}]
			\item There exist $W_1>0$ and $W_2>0$ such that for all $x\in\Lambda$ and $v \in \mathrm{T}_x\mathcal{M} \cap \overline{\mathcal{B}}(0_n,\delta)$
			\begin{align}
				\label{ieq:Retr_property1} \|\mathrm{Retr}_x(v) - x\| &\leq W_1 \|v\|,	\\
				\label{ieq:Retr_property2} \|\mathrm{Retr}_x(v) - x - v\| &\leq W_2\|v\|^2.
			\end{align}
			\item Let $\varphi:\mathbb{R}^n\to\mathbb{R}$ be locally Lipschitz continuous.  ex exists $L_{\varphi}>0$, such that for all $x\in\Lambda$ and $v\in\mathrm{T}_x\mathcal{M} \cap \overline{\mathcal{B}}(0_n,\delta)$,
			\begin{equation} \label{ieq:Retr_property_item2_3}
				| \varphi(\mathrm{Retr}_x(v))-\varphi(x+v) | \leq L_{\varphi}\|v\|^2.
			\end{equation}
			Suppose in addition that $\varphi$ is differentiable with a locally Lipschitz continuous gradient $\nabla \varphi$. Then there exists $L_{\nabla \varphi}>0$ such that for all $x\in\Lambda$ and $v \in \mathrm{T}_x\mathcal{M} \cap \overline{\mathcal{B}}(0_n,\delta)$,
			\begin{equation} \label{ieq:Retr_property_item2_4}
				\begin{aligned}
					| \varphi(x+v)-\varphi(x)-\langle \nabla\varphi(x),v \rangle | &\leq \frac{L_{\nabla \varphi}}{2} \|v\|^2,	\\
					| \varphi(\mathrm{Retr}_x(v))-\varphi(x)-\langle \nabla\varphi(x),v \rangle | &\leq \frac{L_{\nabla \varphi}+2L_\varphi}{2} \|v\|^2.
				\end{aligned}
			\end{equation}
		\end{enumerate}
	\end{lemma}
	
	\begin{proof}
		Item (i) can be obtained directly following a similar line of arguments to the analysis of \cite[Appendix B]{Boumal-Absil-Cartis:JNA:2019}. Here we omit the proof for brevity.
		
		We next prove item (ii). Let $\tilde{\Lambda} := \{ z\in\mathbb{R}^n:\mathrm{dist}(z,\Lambda)\leq\max\{1,W_1\}\delta \}$, where $W_1>0$ is defined in item (i). Clearly, for all $x\in\Lambda$ and $v\in\mathrm{T}_x\mathcal{M} \cap \overline{\mathcal{B}}(0_n,\delta)$, $x+v\in\tilde{\Lambda}$ and $\mathrm{Retr}_x(v)\in\tilde{\Lambda}$ thanks to \eqref{ieq:Retr_property1}. Also, the compactness of $\tilde{\Lambda}$ follows from that of $\Lambda$ is compact. In view of these facts and the local Lipschitz continuity of $\varphi$ on $\mathbb{R}^n$, we see that there exists a constant $C_1>0$, such that for all $x\in\Lambda$ and $v \in \mathrm{T}_x\mathcal{M} \cap \overline{\mathcal{B}}(0_n,\delta)$,
		\begin{equation} \label{proof:lemma:Retr_property:item2:eq1}
			| \varphi(\mathrm{Retr}_x(v))-\varphi(x+v) | \leq C_1 \|\mathrm{Retr}_x(v)-x-v\|.
		\end{equation}
		Invoking \eqref{proof:lemma:Retr_property:item2:eq1}, \eqref{ieq:Retr_property2} and letting $L_{\varphi}=C_1W_2$, we immediately show \eqref{ieq:Retr_property_item2_3}. Suppose in addition $\varphi$ is differentiable with a locally Lipschitz continuous gradient $\nabla \varphi$. Using again the compactness of $\tilde{\Lambda}$, we know that there exists a constant $L_{\nabla \varphi}>0$, such that for all $x\in\Lambda$ and $v \in \mathrm{T}_x\mathcal{M} \cap \overline{\mathcal{B}}(0_n,\delta)$,
		\begin{equation}	\label{proof:lemma:Retr_property:item2:eq2}
			| \varphi(x+v)-\varphi(x)-\langle \nabla \varphi(x),v \rangle | \leq \frac{L_{\nabla \varphi}}{2}\|v\|^2.
		\end{equation}
		
		Combining \eqref{ieq:Retr_property_item2_3} with \eqref{proof:lemma:Retr_property:item2:eq2}, we finally prove \eqref{ieq:Retr_property_item2_4}.
	\end{proof}

	\subsection{Generalized subdifferentials and directional derivative.}
	
	An extended-real-value function $\varphi:\mathbb{R}^n\to (-\infty,+\infty]$ is said to be proper if its domain $\mathrm{dom}(\varphi):=\{x\in\mathbb{R}^n:\varphi(x)<+\infty\}$ is nonempty. A proper function $\varphi$ is said to be closed if $\varphi$ is lower semicontinuous on $\mathbb{R}^n$. For a proper function $\varphi$, its Fr$\mathrm{\acute{e}}$chet and limiting subdiffferential at $x\in\mathrm{dom}(\varphi)$ are defined respectively by
	\begin{displaymath}
		\begin{aligned}
			\widehat{\partial} \varphi(x) &:= \left\{ y\in\mathbb{R}^n: \lim \limits_{z\to x \atop z \neq x} \inf \frac{\varphi(z)-\varphi(x)-\langle y,z-x  \rangle}{\|z-x\|} \geq 0 \right\}, \\
			\partial \varphi(x) &:= \left\{ y\in\mathbb{R}^n: \exists x^k\to x,~\varphi(x^k)\to\varphi(x),~y^k\in\widehat{\partial}\varphi(x^k)~ with~y^k\to y\right\}.
		\end{aligned}
	\end{displaymath}
	We define $\mathrm{dom}(\partial \varphi):=\{x\in\mathrm{dom}(\varphi):\partial\varphi(x)\neq \emptyset\}$. It is straightforward to verify that $\widehat{\partial}\varphi(x) \subseteq \partial\varphi(x)$, $\widehat{\partial}(\alpha \varphi)(x)=\alpha\widehat{\partial}\varphi(x)$, $\partial(\alpha \varphi)(x)=\alpha\partial\varphi(x)$ holds for any $x\in\mathrm{dom}(\varphi)$ and $\alpha >0$. When the function $\varphi$ is convex, both Fr$\mathrm{\acute{e}}$chet subdifferential and limiting subdifferential coincide with the classical subdifferential at any $x\in\mathrm{dom}(\varphi)$ (\cite[Proposition 8.12]{Rockafellar-Wets:VariationalAnalysis:2009}), i.e.,
	\begin{displaymath}
		\widehat{\partial}\varphi(x) = \partial\varphi(x)=\{ y\in\mathbb{R}^n:\varphi(z)-\varphi(x)-\langle y,z-x \rangle\geq 0,~\forall z\in\mathbb{R}^n \}.
	\end{displaymath}
	It is known that $\widehat{\partial}\varphi(x)=\{\nabla \varphi(x)\}$ if $\varphi$ is differentiable at $x$. We say $\varphi$ is continuously differentiable at $x$, if $\varphi$ is differentiable on some $\mathcal{B}(x)$ and $\nabla \varphi$ is continuous at $x$. It can be verify that $\partial\varphi(x)=\{\nabla\varphi(x)\}$ holds if $\varphi$ is continuously differentiable at $x$. 
	
	Recall that the directional derivative, if it exists, of a function $\varphi:\mathbb{R}^n\to\mathbb{R}$ at $x\in\mathbb{R}^n$ along a direction $\eta\in\mathbb{R}^n$ is defined by \cite[Definition 1.1.3]{Cui-Pang:SIAM:2021}
	\begin{displaymath}
		\varphi^\prime(x;\eta) = \lim \limits_{t \searrow 0} \frac{\varphi(x+t\eta)-\varphi(x)}{t}.
	\end{displaymath}
	If $\varphi$ is locally Lipschitz continuous at $x$, then the Clarke generalized derivative of $\varphi$ at $x$ in the direction $\eta$, denoted by $\varphi^{\circ}(x; \eta)$, is defined as follows \cite{Clarke-Ledyaev-Stern-Wolenski:1998}:
	\begin{displaymath}
		\varphi^\circ(x;\eta) := \limsup_{\substack{y \to x \\ t \searrow 0}} \frac{\varphi(y + t\eta) - \varphi(y)}{t}.
	\end{displaymath}
	If $\varphi^\prime(x;\eta)$ exists for all $\eta\in\mathbb{R}^n$, then $\varphi$ is said to be directionally differentiable at $x$. When $\varphi$ is locally Lipschitz continuous around $\overline{x}\in \mathbb{R}^n$ and directionally differentiable at $\overline{x}$, then the statement $\widehat{\partial}\varphi(\overline{x})=\partial\varphi(\overline{x})$ (termed lower regular) is equivalent to the statement $\varphi^\prime(\overline{x};\eta)=\varphi^\circ(\overline{x};\eta)$ for all $\eta \in \mathbb{R}^n$ (termed C(larke)-regular), and is equivalent to the statement \cite[Proposition 4.3.5]{Cui-Pang:SIAM:2021}
	\begin{displaymath}
		\varphi^\prime(\overline{x},\eta) = \sup \limits_{y\in\partial\varphi(\overline{x})} \left\{ \langle \eta,y \rangle \right\},~\forall \eta\in\mathbb{R}^n.
	\end{displaymath}
	
	If $S \subseteq \mathbb{R}^n$ is convex and $\varphi$ is directionally differentiable on $\mathbb{R}^n$, then a necessary condition for a vector $\overline{x}\in S$ to be a local minimizer of $\varphi$ on $S$ is that \cite[Section 2.3]{Cui-Pang:SIAM:2021}
	\begin{equation} \label{ieq:B-stationaryCondition}
		\varphi^\prime(\overline{x};x-\overline{x}) \geq 0 \quad \forall x\in S.
	\end{equation}
	A vector $\overline{x}\in S$ satisfying \eqref{ieq:B-stationaryCondition} is called a directional stationary point of problem $\min\limits_{x\in S} \varphi(x)$ \cite[Section 2.3]{Cui-Pang:SIAM:2021}.
	
	We also need to review some properties of a weakly convex function. Recall that a function $\varphi$ is said to be weakly convex with modulus $\tau>0$ on $\mathbb{R}^n$, if $\varphi+\frac{\tau}{2}\|\cdot\|^2$ is convex on $\mathbb{R}^n$. Thus, it can be easily verified that a weakly convex function $\varphi:\mathbb{R}^n\to\mathbb{R}$ is locally Lipschitz continuous \cite[Corollary 8.40, Theorem 8.38]{Bauschke-Combettes:HilbertSpaces:2017}, and $\bigcup\limits_{x\in S} \partial \varphi(x)$ is nonempty and bounded on any compact set $S \subseteq \mathbb{R}^n$ \cite[Proposition 5.4.2]{Bertsekas:Convex:2009}. For a proper closed convex function $\varphi:\mathbb{R}^n\to (-\infty,+\infty]$, the conjugate $\varphi^{*}$ is also a proper closed convex function and $(\varphi^{*})^{*}=\varphi$ (\cite[Theorem 11.1]{Rockafellar-Wets:VariationalAnalysis:2009}). Moreover, it is known that the following equivalence holds (\cite[Proposition 11.3]{Rockafellar-Wets:VariationalAnalysis:2009}):
	\begin{displaymath}
		\langle x,y \rangle = \varphi(x) + \varphi^{*}(y) \iff y\in\partial \varphi(x) \iff x\in\partial \varphi^{*}(y).
	\end{displaymath}
	
	To end this subsection, we introduce the notion of Clarke subdifferential of a locally Lipschitz function on a Riemannian submanifold of $\mathbb{R}^n$. Let $\mathcal{M}$ be a Riemannian submanifold of $\mathbb{R}^n$, $\varphi:\ \mathcal{M} \rightarrow \mathbb{R}$ be a locally Lipschitz function on $\mathcal{M}$ and $(U, \phi)$ be a chart around $x \in \mathcal{M}$. The Clarke generalized derivative of $\varphi$ at $x$ in the direction $v \in \mathrm{T}_x\mathcal{M}$, denoted by $\varphi^\circ(x;v)$, is defined by \cite{Borckmans-Selvan-Boumal-Absil:JCAM:2014,Liu-Xiao-Yuan:JSC:2024,Si-Absil-Huang-Jiang-Vary:SIOPT:2024}:
	\begin{displaymath}
		\varphi^\circ(x; v) := \limsup_{\substack{y \to x \\ t \searrow 0}} \frac{\hat{\varphi}(\phi(y) + t \mathrm{D}\phi(x)(v)) - \hat{\varphi}(\phi(y))}{t}.
	\end{displaymath}
	where $\hat{\varphi} = \varphi \circ \phi^{-1}$. The Riemannian Clarke subdifferential of $\varphi$ at $x \in \mathcal{M}$, denoted by $\partial_{\scriptscriptstyle C}\varphi (x)$, is defined as
	\begin{displaymath}
		\partial_{\scriptscriptstyle C}\varphi (x): = \{\xi \in \mathrm{T}_x\mathcal{M}:\langle \xi,v\rangle \leq\varphi^{\circ}(x;v),\  \forall v \in \mathrm{T}_x\mathcal{M}\}.
	\end{displaymath}
	It is easy to check by utilizing Proposition 3.1 in \cite{Yang-Zhang-Song:PAC_J_OPTIM:2014} and Proposition 2.2 in \cite{Clarke-Ledyaev-Stern-Wolenski:1998} that $\partial_{\scriptscriptstyle C}(\varphi_1+\varphi_2)(x) \subseteq \partial_{\scriptscriptstyle C}\varphi_1(x)+\partial_{\scriptscriptstyle C}\varphi_2(x)$ and $\partial_{\scriptscriptstyle C}(\varphi_1-\varphi_2)(x) \subseteq \partial_{\scriptscriptstyle C}\varphi_1(x)-\partial_{\scriptscriptstyle C}\varphi_2(x)$. According to \cite[Theorem 4.4]{Yang-Zhang-Song:PAC_J_OPTIM:2014}, a first-order necessary condition for a local minimizer $x^*$ of $\varphi$ on $\mathcal{M}$ is that $0 \in \partial_{\scriptscriptstyle C}\varphi(x^*)$. Recall that a function $\varphi:\mathbb{R}^n \rightarrow \mathbb{R}$ is said to be C(larke)-regular at $x \in \mathcal{M}$ along $\mathrm{T}_x\mathcal{M}$, if $\varphi$ is locally Lipschitz continuous around $x$, $\varphi^\prime(x;v)$ exists and $\varphi^\prime(x;v)= \varphi^{\circ}(x;v)$ for all $v \in \mathrm{T}_x\mathcal{M}$ \cite{clarke1990optimization}. In view of Theorem 5.1 in \cite{Yang-Zhang-Song:PAC_J_OPTIM:2014}, we have $\partial_{\scriptscriptstyle C}\varphi|_{\scriptscriptstyle\mathcal{M}}(x)=  \text{Proj}_{\scriptscriptstyle \mathrm{T}_x\mathcal{M}}\partial_{\scriptscriptstyle C}\varphi(x)$ for $x \in \mathcal{M}$ if $\varphi$ is C-regular at $x$ along $\mathrm{T}_x\mathcal{M}$, where $\mathrm{Proj}_{\mathrm{T}_x\mathcal{M}}$ denotes the orthogonal projection onto $\mathrm{T}_x\mathcal{M}$. Furthermore, if $\varphi$ is differentiable at $x \in \mathcal{M}$, then 
	\begin{displaymath}
		\text{grad}\varphi|_{\scriptscriptstyle \mathcal{M}}(x)=\text{Proj}_{\scriptscriptstyle \mathrm{T}_x\mathcal{M}}\nabla\varphi(x),
	\end{displaymath}
	where $\text{grad}\varphi|_{\scriptscriptstyle \mathcal{M}}$ is the Riemannian gradient of $\varphi|_{\scriptscriptstyle\mathcal{M}}$.

	\subsection{KL property.}
	We now recall KL property, which has been used extensively in the convergence analysis of various first-order methods.
	\begin{definition} \label{defenition:KL_property}(KL property and exponent \cite{Attouch-Bolte-Redont-Soubeyran:MOR:2010}) A proper function $\varphi:\mathbb{R}^n\to (-\infty,+\infty]$ is said to satisfy the KL property at $x\in\mathrm{dom}(\partial \varphi)$ if there exist $\epsilon\in (0,+\infty]$, $\delta>0$, and a continuous concave function $\phi: [0,\epsilon)\to \mathbb{R}_{+}:= [0,+\infty)$, such that:			
		\begin{enumerate}[label = {\upshape(\roman*)}]
			\item $\phi(0)=0$;
			\item $\phi$ is continuously differentiable on $(0,\epsilon)$ with $\phi^\prime>0$;
			\item For any $z\in \mathcal{B}(x,\delta)\cap \{z\in\mathbb{R}^n:\varphi(x)<\varphi(z)<\varphi(x)+\varepsilon\}$, there holds $\phi^\prime\left(\varphi(z)-\varphi(x)\right) \mathrm{dist}(0,\partial\varphi(z))>1$. 
		\end{enumerate}
		If $\varphi$ satisfies the KL property at $x\in\mathrm{dom}(\partial\varphi)$ and $\phi$ can be chosen as $\phi(v) = a_0v^{1-\theta}$ for some $a_0>0$ and $\theta\in[0,1)$, then we say that $\varphi$ satisfies the KL property at $x$ with exponent $\theta$. 
	\end{definition}
	
	A proper function $\varphi:\mathbb{R}^n\to (-\infty,+\infty]$ is called a KL function if it satisfies the KL property at any point in $\mathrm{dom}(\partial \varphi)$. For connections between the KL property and the well-known error bound theory \cite{Pang:MP:1997}, we refer the interested readers to \cite{Bolte-Nguyen-Peypouquet-Suter:MP:2017}. It is well-known that semialgebraic functions provide a very rich subclass of KL functions. A extended-valued function is called semialgebraic if its graph is a semialgebraic set, which can be described as unions or intersections of finitely many lower sets of polynomials \cite[Definition 2.1]{Attouch:MP:2013}. A proper lower semicontinuous semialgebraic function is a KL function with an exponent $\theta \in[0,1)$. The semialgebraic function covers many commonly used functions in optimization problems that arise from applications, for example, maximum or minimum of finitely many polynomials, Euclidean norms and eigenvalues. Also, sums, products, quotients of semialgebraic functions and indicator functions on semialgebraic sets are still semialgebraic functions. The following lemma reveals a useful result that a KL function, if being constant on a compact set, has a uniformized expression on the KL property.
	
	\begin{lemma}{$($Uniformized KL property \cite[Lemma 6]{Bolte-Sabach-Teboulle:MP:2014}$)$} \label{lemma:Uniformized_KL_property}
		Let $\Upsilon \subseteq \mathbb{R}^n$ be a compact set, and the proper function $\varphi:\mathbb{R}^n\to (-\infty,+\infty]$ be constant on $\Upsilon$. If $\varphi$ satisfies the KL property at each point of $\Upsilon$, then there exist $\varepsilon$, $\delta>0$ and a continuous concave function $\phi:[0,\epsilon)\to [0,+\infty)$ satisfying Definition~\ref{defenition:KL_property} (i) and (ii) such that
		\begin{displaymath}
			\phi^\prime\left(\varphi(z)-\varphi(x)\right)\mathrm{dist}(0_n,\partial\varphi(z)) \geq 1
		\end{displaymath}
		holds for any $x\in\Upsilon$ and $z\in \mathcal{B}(x,\delta)\cap \{z\in\mathbb{R}^n:\varphi(x)<\varphi(z)<\varphi(x)+\varepsilon\}$.
	\end{lemma}
	
\section{Stationary points of problem \eqref{problemn1}.}
	In this section, we introduce four different notions of stationary point for problem \eqref{problemn1} and investigate 	their relationships.

	\subsection{B-stationary and Riemannian C-stationary points of a function on Riemannian submanifolds.}
	In this subsection, we investigate B-stationarity and Riemannian C-stationarity for a locally Lipschitz continuous function on Riemannian submanifolds of $\mathbb{R}^n$. Recall that the manifold $\mathcal{M}$ in this paper is assumed to be a Riemannian submanifold of $\mathbb{R}^n$. We begin with the definition of B(ouligand)-stationarity and Riemannian C(larke)-stationarity for a locally Lipschitz continuous function over $\mathcal{M}$.
	\begin{definition}
		Let $\varphi$ be a locally Lipschitz continuous and real-valued function defined on an open set containing the manifold $\mathcal{M}$. 
		\begin{enumerate}[label = {\upshape(\roman*)}]
			\item \cite[Definition 6.1.1]{Cui-Pang:SIAM:2021} We say $\overline{x}\in\mathcal{M}$ is a B(ouligand)-stationary point of $\varphi$ on $\mathcal{M}$ if for all $v\in\mathrm{T}_{\overline{x}}\mathcal{M}$, $\varphi^\prime(\overline{x};v)$ exists and satisfies $\varphi^\prime(\overline{x};v)\geq0$.
			\item \cite{Yang-Zhang-Song:PAC_J_OPTIM:2014} We say $\overline{x}\in\mathcal{M}$ is a Riemannian C(larke)-stationary point of $\varphi$ on $\mathcal{M}$ if $0_n\in\partial_C \varphi\vert_\mathcal{M}(\overline{x})$.
		\end{enumerate}
	\end{definition}
	
	Next we review the relationships between a local minimizer of $\varphi$ on $\mathcal{M}$ and a B-stationary point as well as a Riemannian C-stationary point of $\varphi$ on $\mathcal{M}$. Item (\romannumeral1) in the following proposition is from the statement below \cite[Definition 6.1.1]{Cui-Pang:SIAM:2021}, while Item (\romannumeral2) is a direct result of \cite[Proposition 3.3]{Yang-Zhang-Song:PAC_J_OPTIM:2014} and \cite[(3.8)]{Yang-Zhang-Song:PAC_J_OPTIM:2014}.
	
	\begin{proposition}  \label{Prop:B-stationaryPoint}
		Let $\mathcal{O} \supseteq \mathcal{M}$ be an open set of $\mathbb{R}^n$ and $\varphi:\mathcal{O}\to \mathbb{R}$ be locally Lipschitz continuous. Suppose $\overline{x}\in\mathcal{M}$ is a local minimizer of $\varphi$ on $\mathcal{M}$.
		\begin{enumerate}[label = {\upshape(\roman*)}]
			\item If $\varphi^\prime(\overline{x};v)$ exists for all $v \in \mathrm{T}_{\overline{x}} \mathcal{M}$, then $\overline{x}$ is a B-stationary point of $\varphi$ on $\mathcal{M}$.
			\item $\overline{x}$ is a Riemannian C-stationary point of $\varphi$ on $\mathcal{M}$.
		\end{enumerate}
		
	\end{proposition}
	
	In Euclidean space, we know that for a C-regular function the B-stationarity and C-stationarity are equivalent \cite[Proposition 4.3.5]{Cui-Pang:SIAM:2021}. If the function is not C-regular, B-stationarity implies C-stationarity when the underlying function is locally Lipschitz continuous and directional differentiable. However, the converse implication of C-stationarity implying B-stationarity is not always valid \cite{Pang-Razaviyayn-Alvarado:MATH_OPER_RES:2017}. Similar results hold for nonsmooth optimization on manifolds, as presented in the following proposition.
	\begin{proposition}
		Let $\mathcal{O} \supseteq \mathcal{M}$ be an open set of $\mathbb{R}^n$ and $\varphi:\mathcal{O}\to\mathbb{R}$ be locally Lipschitz continuous and directional differentiable at each $x\in \mathcal{M}$.
		\begin{enumerate}[label = {\upshape(\roman*)}]
			\item If $\overline{x}\in\mathcal{M}$ is a B-stationary point of $\varphi$ on $\mathcal{M}$ then $\overline{x}$ is a Riemannian C-stationary point of $\varphi$ on $\mathcal{M}$.
			\item If $\varphi$ is C-regular at $\overline{x}\in \mathcal{M}$ along $\mathrm{T}_{\overline{x}} \mathcal{M}$, then the statements that $\overline{x}$ is a B-stationary point and Riemannian C-stationary point of $\varphi$ on $\mathcal{M}$ are equivalent.
		\end{enumerate}
	\end{proposition}
	\begin{proof}
		We first prove Item (\romannumeral1). Denote by $\overline{\varphi}$ the restriction of $\varphi$ on $\mathcal{M}$, i.e., $\overline{\varphi}:= \varphi\vert_{\mathcal{M}}$. It suffices to show $\overline{\varphi}^\circ(\overline{x};v) \geq \varphi^\prime (\overline{x};v)$ for all $v\in \mathrm{T}_{\overline{x}} \mathcal{M}$, since the statements that ${\rm 0}\in\partial_C \overline{\varphi}(\overline{x})$ and $\overline{\varphi}^\circ (\overline{x};v)\geq 0$ for all $v\in \mathrm{T}_{\overline{x}} \mathcal{M}$ are equivalent due to the definition of $\partial_C \overline{\varphi}$. Let $(U,\phi)$ be a chart of $\mathcal{M}$ around $\overline{x}$. Then, there exist $t_1>0$ and $\delta>0$ such that for any $t\in\left[0,t_1\right]$, $y\in \mathcal{M}$ with $\|y-\overline{x}\|<\delta$ there holds \cite[Theorem 5.1]{Yang-Zhang-Song:PAC_J_OPTIM:2014} $$\phi^{-1}(\phi(y)+t\mathrm{D} \phi(x)(v))=y+t\psi(t,y),$$
		where $\psi(t,y)=\mathrm{D} \phi^{-1}(\phi(y))(\mathrm{D} \phi(\overline{x})(v)) + t\tilde{\phi}(y)$ with $\tilde{\phi}$ being smooth of $y$. Since $\mathrm{D} \phi( \overline{x})$ is invertible, one has $\lim \limits_{t \searrow 0} \psi(t,\overline{x})=v$ by invoking \cite[8(19)]{Rockafellar-Wets:VariationalAnalysis:2009}. According to \cite[Proposition 4.1.3]{Cui-Pang:SIAM:2021}, $\varphi$ is semidifferentiable at $\overline{x}$, i.e., $\lim_{\substack{w\rightarrow v\\ t \searrow 0}}\frac{\varphi(\overline{x}+ tw)- \varphi(\overline{x})}{t}$ exists and equals to $\varphi^\prime(\overline{x};v)$. These together with the definition of $\overline{\varphi}^\circ(\overline{x}; v)$ yield that
		\begin{align}
			\overline{\varphi}^\circ(\overline{x};v)& = \limsup_{\substack{t \searrow 0\\y \to \overline{x}}} \frac{\varphi(y+t\psi(t,y))-\varphi(y)}{t}\notag \\
			&\geq \limsup_{t\searrow0} \frac{\varphi(\overline{x} +t\psi(t,\overline{x}))-\varphi(\overline{x})}{t} \notag\\
			&= \lim_{\substack{t\searrow 0 \\ v^\prime\rightarrow v} }\frac{\varphi(\overline{x}+t v^\prime)-\varphi(\overline{x})}{t}\notag\\
			&=\varphi^\prime(\overline{x};v).\notag
		\end{align}
		Then, item (\romannumeral1) follows immediately.
		\par We next prove Item (\romannumeral2). Since $\varphi$ is C-regular at $\overline{x}\in\mathcal{M}$ along $\mathrm{T}_{\overline{x}} \mathcal{M}$, we have $\partial_C \varphi\vert_{\mathcal{M}} (\overline{x})={\rm Proj}_{\mathrm{T}_{\overline{x}} \mathcal{M}} \partial_C \varphi(\overline{x})$ \cite[Theorem 5.1]{Yang-Zhang-Song:PAC_J_OPTIM:2014}. Thus, the statements that ${\rm 0_n}\in\partial_C \varphi \vert_{\mathcal{M}}(\overline{x})$ and ${\rm 0_n} \in \partial_C \varphi(\overline{x})+\mathrm{N}_{\overline{x}} \mathcal{M}$ are equivalent. Invoking \cite[(4.22)]{Cui-Pang:SIAM:2021}, the latter one is equivalent to $\varphi^\circ(\overline{x};v) \geq 0$ for all $v\in \mathrm{T}_{\overline{x}} \mathcal{M}$. In view of the C-regularity of $\varphi$ at $\overline{x}$ along $\mathrm{T}_{\overline{x}} \mathcal{M}$, we obtain that $\varphi^\circ(\overline{x};v)=\varphi^\prime(\overline{x};v)$ for all $v\in \mathrm{T}_{\overline{x}} \mathcal{M}$. These imply (\romannumeral2) immediately.
	\end{proof}\vspace{-5pt}
	
	\subsection{Relationships among different stationary points of problem \eqref{problemn1}.}
	In this subsection, we first give specific characterizations on the B-stationarity of problem \eqref{problemn1}. Based on these characterizations, we further introduce another two concepts of stationarity for problem \eqref{problemn1}. Finally, we examine the relationships among the B-stationary points, Riemannian C-stationary points, lifted B-stationary points and critical points of problem \eqref{problemn1}.
	
	We begin with an equivalent characterization for B-stationary point of problem \eqref{problemn1}.
	\begin{proposition} \label{prop23_dStationary}
		A vector $\overline{x} \in \mathcal{M}$ is a B-stationary point for problem \eqref{problemn1} if and only if
		\begin{equation} \label{proposition:d_stationary}
			\partial h_2(\overline{x}) + \frac{\partial f(\overline{x})}{g(\overline{x})} \subseteq \partial h_1(\overline{x}) + \nabla r(\overline{x}) + \frac{f(\overline{x})}{g^2(\overline{x})} \nabla g(\overline{x}) + \mathrm{N}_{\overline{x}}\mathcal{M};
		\end{equation}
		or equivalently,
		\begin{equation} \label{proposition:d_stationary2}
			\partial_Ch_2\vert_{\mathcal{M}}(\overline{x}) + \frac{\partial_Cf\vert_{\mathcal{M}}(\overline{x})}{g(\overline{x})}\subseteq \partial_C h_1\vert_{\mathcal{M}}(\overline{x})+ {\rm grad} r\vert_{\mathcal{M}} (\overline{x}) + \frac{f(\overline{x})}{g^2( \overline{x})}{\rm grad} g\vert_{\mathcal{M}}(\overline{x}).
		\end{equation}
	\end{proposition}
	\begin{proof}
		Let $\overline{S}\subseteq \mathbb{R}^n$ be defined by
		\begin{equation} \label{proof:proposition:d_stationary:S_overline}
			\overline{S} := \partial h_2(\overline{x}) + \frac{\partial f(\overline{x})}{g(\overline{x})} - \nabla r(\overline{x}) - \frac{f(\overline{x})}{g^2(\overline{x})} \nabla g(\overline{x}).
		\end{equation}
		It is clear that $\overline{S}$ is compact and $F$ is locally Lipschitz continuous on an open set containing $\mathcal{M}$, thanks to Assumption~\ref{ass1}. We also obtain from Assumption~\ref{ass1} that the directional derivatives $h_1'(x;v)$, $h_2'(x;v)$, $r'(x;v)$, $g'(x;v)$ and $f'(x;v)$ are all well defined for all $x\in\mathcal{M}$ and $v\in\mathbb{R}^n$. According to the definition, $\overline{x}\in\mathcal{M}$ is a B-stationary point of problem \eqref{problemn1} if for any $v\in\mathrm{T}_{\overline{x}}\mathcal{M}$, $F^\prime(\overline{x};v)\geq 0$. This is equivalent to
		\begin{displaymath}
			\begin{aligned}
				h_1^\prime(\overline{x};v) &\geq \left( h_2+\frac{1}{g(\overline{x})}f-r-\frac{f(\overline{x})}{g^2(\overline{x})}g \right)^\prime(\overline{x};v)	\\
				&= \max \limits_{y\in\overline{S}}~ \langle y,v \rangle,
			\end{aligned}
		\end{displaymath}
		where the equality holds due to the (weakly) convexity of $h_2$ and $f$, the smoothness of $r$ and $g$, as well as the compactness of $\overline{S}$. The above display indicates that $\overline{x}\in\mathcal{M}$ is a B-stationary point of $F$ if and only if for any $v\in\mathrm{T}_{\overline{x}}\mathcal{M}$ and $y\in\overline{S}$, $h_1^\prime(\overline{x};v) \geq \langle y,v \rangle$, which is also equivalent to the statement that 
		\begin{equation}	\label{proof:proposition:d_stationary:star2}
			0_n \in \mathop{\arg\min}\limits_{v\in\mathrm{T}_{\overline{x}}\mathcal{M}}\left\{ h_1(\overline{x}+v)-\langle y,v \rangle \right\}, ~ \forall y\in\overline{S},
		\end{equation}
		thanks to the convexity of $h_1$. In view of the Fermat's rule, the inclusion \eqref{proof:proposition:d_stationary:star2} holds if and only if
		\begin{equation} \label{proof:proposition:d_stationary:star3}
			y\in \partial h_1(\overline{x}) + \mathrm{N}_{\overline{x}}\mathcal{M}, ~ \forall y\in\overline{S}.
		\end{equation}
		Combining \eqref{proof:proposition:d_stationary:S_overline} and \eqref{proof:proposition:d_stationary:star3}, we immediately obtain the inclusion \eqref{proposition:d_stationary}. Finally, \eqref{proposition:d_stationary2} follows from the C-regularity of $h_1,h_2,f,g$ and $r$ as well as \cite[Theorem 5.1]{Yang-Zhang-Song:PAC_J_OPTIM:2014}.
	\end{proof}
	
	Inspired by Proposition~\ref{prop23_dStationary}, we further introduce the notions of lifted B-stationary points and critical points.
	
	\begin{definition} \label{definition:lifted_sta_And_critical}
		For problem \eqref{problemn1}, we say that $\overline{x}\in\mathcal{M}$ is 
		\begin{enumerate}[label = {\upshape(\roman*)}]
			\item a lifted B-stationary point if there exists $\overline{y}\in\partial f(\overline{x})$ such that
			\begin{equation} \label{definition:lifted_d_stationary_point}
				\partial h_2(\overline{x}) + \frac{\overline{y}}{g(\overline{x})} \subseteq \partial h_1(\overline{x}) + \nabla r(\overline{x}) + \frac{f(\overline{x})}{g^2(\overline{x})}\nabla g(\overline{x}) + \mathrm{N}_{\overline{x}}\mathcal{M};
			\end{equation}
			\item a critical point if 
			\begin{equation} \label{definition:critical_point}
				0_n \in \partial h_1(\overline{x}) - \partial h_2(\overline{x}) + \nabla r(\overline{x}) + \frac{f(\overline{x})}{g^2(\overline{x})}\nabla g(\overline{x}) - \frac{\partial f(\overline{x})}{g(\overline{x})} + \mathrm{N}_{\overline{x}}\mathcal{M}.
			\end{equation}
		\end{enumerate}
	\end{definition}
	
	By the C-regularity of $h_1,h_2,f,g$ and $r$, we obtain the following corollary immediately from \cite[Theorem 5.1]{Yang-Zhang-Song:PAC_J_OPTIM:2014}
	\begin{corollary}\label{corollary*}
		For problem \eqref{problemn1}, $\overline{x} \in \mathcal{M}$ is
		\begin{enumerate}[label = {\upshape(\roman*)}]
			\item a lifted B-stationary point if and only if
			$$\partial_C h_2\vert_{\mathcal{M}}(\overline{x}) + \frac{\overline{y}^\prime}{g(\overline{x})} \subseteq \partial_C h_1\vert_{\mathcal{M}}(\overline{x}) + {\rm grad}r\vert_{\mathcal{M}} (\overline{x})+ \frac{f(\overline{x})}{g^2 (\overline{x})}{\rm grad}g\vert_{\mathcal{M}}( \overline{x})$$ 
			for some ${\overline{y}}^\prime \in \partial_C f\vert_{\mathcal{M}}(\overline{x})$;
			\item a critical point if and only if $$ 0_n \in \partial_C h_1\vert_{\mathcal{M}}( \overline{x})-\partial_C h_2\vert_{\mathcal{M}}( \overline{x})+{\rm grad}r\vert_{\mathcal{M}}(\overline{x}) +\frac{f(\overline{x})}{g^2 (\overline{x})}{\rm grad}g\vert_{\mathcal{M}}( \overline{x})-\frac{\partial_C f\vert_{\mathcal{M}} ( \overline{x})}{g(\overline{x})}.$$
		\end{enumerate}
	\end{corollary}

	To close this section, we finally examine the relationships among B-stationary points, Riemannian C-stationary points, lifted B-stationary points and critical points in the next proposition. These relationships are also illustrated in Figure \ref{fig1}.
	\begin{figure}[h]
		\centering
		\includegraphics[width=1\linewidth]{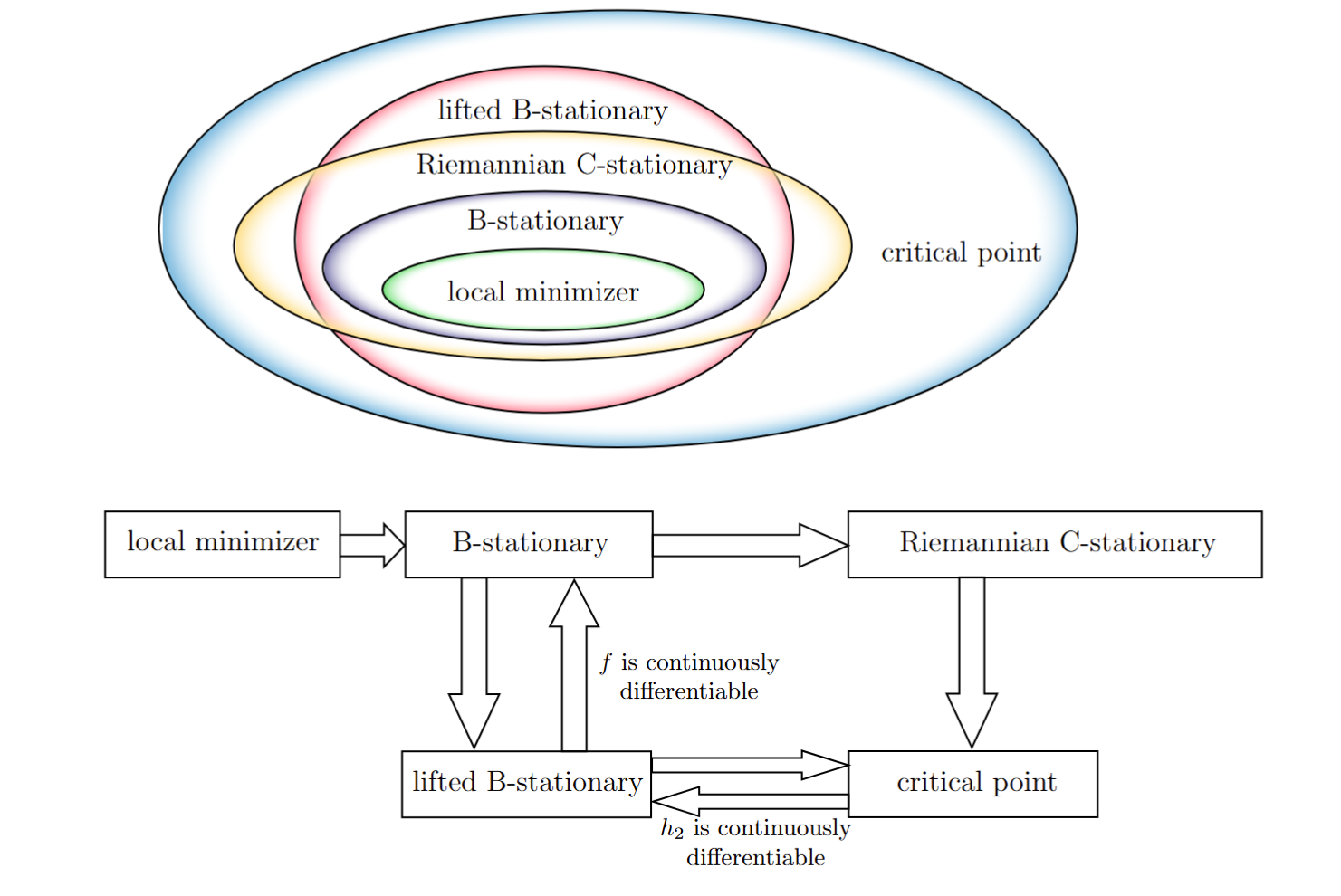}
		\caption{The relationships among local minimizers and various stationary points of problem \eqref{problemn1}.}
		\label{fig1}
	\end{figure}
	
	\begin{proposition}	\label{proposition_28_definition}
		Consider problem \eqref{problemn1}. Then the following statements hold:
		\begin{enumerate}[label = {\upshape(\roman*)}]
			\item If $x^*\in\mathcal{M}$ is a local minimizer of problem \eqref{problemn1}, then it is a B-stationary point of problem \eqref{problemn1};
			\item If $x^{*} \in \mathcal{M}$ is a B-stationary point for problem \eqref{problemn1}, then it is a lifted B-stationary point and a Riemannian C-stationary point of problem \eqref{problemn1};
			\item If $x^{*} \in \mathcal{M}$ is a lifted B-stationary point of problem \eqref{problemn1}, then it is a critical point of problem \eqref{problemn1};
			\item If $x^{*} \in \mathcal{M}$ is a Riemannian C-stationary point of problem \eqref{problemn1}, then it is a critical point of problem \eqref{problemn1};
			\item When f is continuously differentiable at $x^{*} \in \mathcal{M}$, $x^{*}$ is a B-stationary point of problem \eqref{problemn1} if and only if $x^{*}$ is a lifted B-stationary point of problem \eqref{problemn1};
			\item When $h_2$ is continuously differentiable at $x^*\in\mathcal{M}$, $x^*$ is a lifted B-stationary point of problem \eqref{problemn1} if and only if $x^*$ is a critical point of problem \eqref{problemn1};
		\end{enumerate}
	\end{proposition}
	\begin{proof}
		Item (\romannumeral2)-(\romannumeral3) and (\romannumeral5)-(\romannumeral6) are direct consequence of Proposition \ref{prop23_dStationary}, Definition \ref{definition:lifted_sta_And_critical} and Corollary \ref{corollary*}. Item (\romannumeral1) follows from Proposition \ref{Prop:B-stationaryPoint}. We next prove Item (\romannumeral4). It suffices to prove
		\begin{align}
			\label{pro3.9_pf_1}
			\partial_C\frac{f\vert_{\mathcal{M}}}{g\vert_{\mathcal{M}}} (\overline{x}) = \frac{\partial_C f\vert_{\mathcal{M}}(\overline{x})}{g(\overline{x})} - \frac{f(\overline{x})}{g^2(\overline{x})} {\rm grad} g\vert_{\mathcal{M}} (\overline{x}),
		\end{align}
		since we know $\partial_CF\vert_\mathcal{M}( \overline{x})\subseteq\partial_C h_1\vert_{\mathcal{M}} (\overline{x}) - \partial_C h_2\vert_{\mathcal{M}} (\overline{x}) + {\rm grad} r\vert_{\mathcal{M}}(\overline{x}) + \partial_C\frac{f\vert_{\mathcal{M}}}{ g\vert_{\mathcal{M}}}(\overline{x})$. In order to prove \eqref{pro3.9_pf_1}, we first show $f/g$ is C-regular at $\overline{x}$. For all $x\in \mathcal{M}$ and $v\in \mathbb{R}^n$, we have
		\begin{align}
			\left(\frac{f}{g} \right)^\circ(x;v) &= \limsup_{\substack{y\rightarrow x \\ \tau \searrow 0}}\frac{\frac{f(y +\tau v)}{g(y +\tau v)}-\frac{f(y)}{g(y)}}{\tau}\notag\\
			&=\limsup_{\substack{y\rightarrow x \\ \tau \searrow 0}}\frac{g(y)(f(y+\tau v)-f(y)) - f(y)(g(y+\tau v)-g(y))}{\tau g(y+\tau v)g(y)}.\notag
		\end{align}
		Thanks to the continuous differentiability of $g$ and the continuity of $f$, it follows from the above equality that
		\begin{displaymath}
			\begin{aligned}
				\left(\frac{f}{g} \right)^\circ(x;v)& = \limsup_{\substack{y\rightarrow x \\ \tau \searrow 0}}\frac{g(y)(f(y+\tau v)-f(y))}{\tau g(y+\tau v)g(y)}-\frac{f(x)}{g^2(x)} \langle\nabla g(x),v \rangle\\
				&=\frac{f^\prime(x;v)}{g(x)}-\frac{f(x)}{ g^2(x)}g^\prime(x;v)\\
				&=\left(\frac{f}{g}\right)^\prime(x;v),
			\end{aligned}
		\end{displaymath}
		where the second equation follows from the C-regularity of $f$ and the continuity of $g$. This leads to the C-regularity of $\frac{f}{g}$. Therefore, we get $\partial_C\frac{f}{g}(\overline{x})= \partial \frac{f}{g}(\overline{x}) = \frac{1}{g( \overline{x})}\partial f(\overline{x}) - \frac{f(\overline{x})}{g^2(\overline{x})} \nabla g(\overline{x})$ by invoking \cite[Corollary 1.111]{Mordukhovich:VariationalAnalysis_I:2006}, the C-regularity of $f(\overline{x})$ and the positivity of $g(\overline{x})$, as well as the Lipschitz differentiability of $g$. This yields that $\partial_C\frac{f}{g}( \overline{x})=\frac{1}{g(\overline{x})}\partial_C f(\overline{x})-\frac{f(\overline{x})}{g^2 (\overline{x})}\nabla g(\overline{x})$. In view of the C-regularity of $f/g$ and $f$, \eqref{pro3.9_pf_1} follows immediately. We complete the proof.
	\end{proof}

\section{A manifold proximal-gradient-subgradient algorithm for solving \eqref{problemn1}.}
	In this section, we propose a manifold proximal-gradient-subgradient algorithm (MPGSA) for solving problem \eqref{problemn1} and establish its subsequential convergence under mild conditions.
	
	Our algorithm is presented as MPGSA in Algorithm \ref{algorithm1}. The design of our MPGSA is motivated by the ManPG algorithm \cite{Chen-Ma-Anthony-Zhang:SIOPT:2020}. Before proceeding we make two remarks as follows:
	\par (\romannumeral1) Given an iterate $x^k\in\mathcal{M}$, step 2 of MPGSA computes a potential descent direction $v^k$ of $F$ restricted to the tangent space $\mathrm{T}_{x^k}\mathcal{M}$ by tackling the strongly convex optimization problem \eqref{eqAlgV}, which occupies the major computational cost of the proposed algorithm. For all $k\in \mathbb{N}$, let $E_k\in\mathbb{R}^{n\times d}$ such that its columns form a basis of the normal space $\mathrm{N}_{X^k}\mathcal{M}$ and thus $\mathrm{T}_{X^k} \mathcal{M}=\{v\in\mathbb{R}^n:{E_k}^Tv=0_d\}$. In the case where $h_1$ is the $\ell_1$ norm, \eqref{eqAlgV} is well-structured and can be efficiently solved to a high accuracy by the semismooth Newton method developed in \cite[Section 4.2]{Chen-Ma-Anthony-Zhang:SIOPT:2020}. Next we are concentrated on solving \eqref{eqAlgV} in a more general case where the proximal operator associated with $h_1$ can be efficiently evaluated\footnote{The proximal operator associated with $h_1$ is defined as $\mathrm{prox}_{h_1}(x)=\mathop{\arg\min}\limits_{y \in \mathbb{R}^n}\{h_1(y)+\frac{1}{2}\| y-x \|^2 \}.$} but $h_1$ is not necessarily the $\ell_1$ norm. Recall from \cite[Definition 12.20]{Bauschke-Combettes:HilbertSpaces:2017} that for a proper, closed and convex function $\varphi:\mathbb{R}^n\rightarrow \overline{\mathbb{R}}$, the Moreau envelope of $\varphi$ with the parameter $\beta>0$ is defined at $x\in\mathbb{R}^n$ as $$\varphi_\beta(x):=\min\limits_{y\in\mathbb{R}^n} \left\{\varphi(y)+ \frac{1}{2\beta}\| y-x \|^2 \right\}.$$ It is not hard to verify that the dual problem of \eqref{eqAlgV} is also convex and can be formulated into 
	\begin{equation} \label{dual_manpg}
		\min\limits_{u\in\mathbb{R}^n} \left\{\frac{t_k}{2} \|E_ku-w^k\|_2^2-(h_1)_{t_k}(t_kE_ku-t_kw^k+x^k) \right\},
	\end{equation}
	where $w^k=\nabla r(x^k)+\frac{f(x^k)\nabla g(x^k)}{g^2 (x^k)}-\frac{y^k}{g(x^k)}-z^k$. In view of \cite[Proposition 2.29]{Bauschke-Combettes:HilbertSpaces:2017}, $(h_1)_{t_k}$ is differentiable and its gradient $\nabla (h_1)_{t_k} := \frac{1}{t_k}(I_n-\mathrm{prox}_{t_kh_1})$ is Lipschitz continuous with the Lipschitz constant $\frac{1}{t_k}$. Hence, one can directly apply the Nesterov's method \cite{nesterov2013gradient} or FISTA \cite{beck2009fast} to problem \eqref{dual_manpg} and consequently obtain a global optimal solution $u^k$. Invoking this and \cite[Proposition 19.3]{Bauschke-Combettes:HilbertSpaces:2017}, it is easy to deduce that the designed $v^k$ can be computed by
	$$v^k=\mathrm{prox}_{t_kh_1}(t_kE_ku^k-t_kw^k +x^k)-x^k.$$
	\par (\romannumeral2) Since for an arbitrary stepsize $\alpha>0$, $x^k+\alpha v^k$ does not necessarily stays in $\mathcal{M}$, in step 3 of MPGSA we need to bring it back to $\mathcal{M}$ through performing a retraction $\mathrm{T}_{x^k}\mathcal{M}$. Also, an Armijo line-search procedure is incorporated to determine the stepsize $\alpha$ and its well-definedness will be shown later on. In particular, compared to the ManPG method \cite[Algorithm 4.1]{Chen-Ma-Anthony-Zhang:SIOPT:2020}, we see that MPGSA reduces to the ManPG method when $h_2=f=0$, $g$ is a positive constant and $\mathcal{M}$ is the Stiefel manifold. 
	\par In what follows, we conduct convergence analysis for MPGSA. To this end, we introduce the following assumption regarding its initial point.

	\begin{assumption} \label{assumption:section4:compactSetX}
		The level set $\mathcal{X} := \{x\in\mathcal{M}: F(x)\leq F(x^0)\}$ is compact. 
	\end{assumption}
	
	Note that the boundedness of the level set associated with the extended objective is a quite standard assumption, while its closedness automatically holds thanks to the continuity of $F$. Specifically, Assumption~\ref{assumption:section4:compactSetX} is fulfilled by the proposed algorithm with an arbitrary initial point $x^0$ when $\mathcal{M}$ is compact. Next, we present two auxiliary Lemmas that will be frequently used in the convergence analysis. For ease of presentation, we define a set-valued mapping $\Lambda$ on $\mathcal{M}$ as
	\begin{equation} \label{eq:section4:mappingLambda}
		\Lambda(x) := \nabla r(x) + \frac{f(x)\nabla g(x)}{g^2(x)} - \frac{\partial f(x)}{g(x)} - \partial h_2(x), \quad \mbox{for }x\in\mathcal{M}
	\end{equation}
	and a mapping $v:\mathcal{M}\times\mathbb{R}^n\times\mathbb{R}_{+}\to\mathbb{R}^n$ as
	\begin{equation} \label{eq:section4:mappingV}
		v(x,w,t) = \mathop{\arg\min} \limits_{v\in \mathrm{T}_x\mathcal{M}} \left\{ h_1(x+v)+\langle w,v \rangle + \frac{1}{2t}\|v\|^2 \right\}.
	\end{equation}
	
	\begin{algorithm}[h]
		\caption{\ ${\rm MPGSA}$ for solving \eqref{problemn1}}\label{algorithm1}
		\begin{algorithmic}
			\State \vspace{-0.15 cm}
			\begin{description}
				\item[\bf Step 0.] Input $x^0 \in \mathcal{M}$, $\gamma \in (0,1)$, $t_k \in (\underline{t},\overline{t})$ with $0<\underline{t}<\overline{t}$. Set $k \leftarrow 0$. \vspace{0.1 cm}
				\item[\bf Step 1.] Choose
				\[
				y^k \in \partial f(x^k), \quad z^k \in \partial h_2(x^k).
				\]	
				\vspace{0.1 cm}
				\item[\bf Step 2.] Compute
				\begin{equation}\label{eqAlgV}
					\begin{aligned}
						v^k := \mathop{\arg\min}\limits_{v \in \mathrm{T}_{x^k}\mathcal{M}} \left\{ h_1(x^k+v) + \Bigg\langle \nabla r(x^k) + \frac{f(x^k) \nabla g(x^k)}{g^2(x^k)} - \frac{y^k}{g(x^k)} \right.\\
						\left. - z^k, v \Bigg\rangle+ \frac{1}{2 t_k} \|v\|^2 \right\}.
					\end{aligned}	
				\end{equation}
				\vspace{0.1 cm}
				\item[\bf Step 3.] For $m=0,1,\dots,$ do
				\begin{itemize}
					\item Set $\alpha_k = \gamma^m$,
					\item Compute $\widetilde{x}^{k+1} = \mathrm{Retr}_{x^k}(\alpha_k v^k)$,
					\item If $F(\widetilde{x}^{k+1}) \leq F(x^{k}) - \frac{\alpha_k}{2 t_k} \|v^k\|^2$, set $x^{k+1} = \widetilde{x}^{k+1}$ and go to Step 4.
				\end{itemize}
				\vspace{0.1 cm}
				\item[\bf Step 4.] Set $k \leftarrow k+1$ and go to Step 1.
			\end{description}
		\end{algorithmic}
	\end{algorithm}
	
	\begin{lemma}	\label{lem1}
		Suppose that Assumption~\ref{assumption:section4:compactSetX} holds. Let $\Lambda$ and $v$ be defined in \eqref{eq:section4:mappingLambda} and \eqref{eq:section4:mappingV}, respectively. Then there exist three positive constants $M_1$, $M_2$ and $M_3$ such that the following statements hold:
		\indent
		\begin{enumerate}[label = {\upshape(\roman*)}]
			\item \label{lem1-1} $\frac{1}{g(x)} \leq M_1$, for all $x\in\mathcal{X}$; 
			\item \label{lem1-2} $\|w\|\leq M_2$, for all $x\in\mathcal{X}$ and $w\in\Lambda(x)$;
			\item \label{lem1-3} $\|v(x,w,t)\| \leq M_3:=2\overline{t}(L_1+M_2)$, for all $x\in\mathcal{X}$, $t\in[\underline{t},\overline{t}]$ and $w\in\Lambda(x)$.
		\end{enumerate}
	\end{lemma}
	
	\begin{proof}
		Item (i) follows directly from $\inf\{g(x):x\in\mathcal{X}\}>0$ thanks to Assumption~\ref{assumption:section4:compactSetX}, Assumption~\ref{ass1} (\romannumeral3) and (\romannumeral4).
		
		We next prove item (ii). First, $\nabla r(x)+f(x)\nabla g(x)/g^2(x)$ is bounded on $\mathcal{X}$ due to item (i), the local Lipschitz continuity of $f(x)$, $\nabla r(x)$ and $\nabla g(x)$, as well as the compactness of $\mathcal{X}$. Second, since $f$ and $h_2$ are (weakly) convex and real-valued, in view of Assumption~\ref{assumption:section4:compactSetX} we deduce that $\partial f(x)$ and $\partial h_2(x)$ are uniformly bounded for all $x\in\mathcal{X}$. Based on the above discussions and the definition of $\Lambda(x)$ in \eqref{eq:section4:mappingLambda}, we obtain item (\romannumeral2).
		
		Finally we prove item (iii). Given $x\in\mathcal{X}$, $t\in[\underline{t},\overline{t}]$ and choose an arbitrary $w\in\Lambda(x)$, we have from \eqref{eq:section4:mappingV} and $0_n\in\mathrm{T}_x\mathcal{M}$ that 
		\begin{displaymath}
			\begin{aligned}
				h_1(x) &\geq h_1(x+v(x,w,t)) - \langle w,v(x,w,t) \rangle + \frac{1}{2t}\|v(x,w,t)\|^2  \\
				&\geq h_1(x+v(x,w,t)) - \|w\| \|v(x,w,t)\| + \frac{1}{2t} \|v(x,w,t)\|^2.
			\end{aligned}
		\end{displaymath}
		In view of Assumption~\ref{ass1} (ii) and item (ii) of this lemma, we further obtain from the above inequality that
		\begin{displaymath}
			\begin{aligned}
				\|v(x,w,t)\|^2 &\leq 2t\left( h_1(x)-h_1(x+v(x,w,t)) + \|w\| \|v(x,w,t)\| \right)\\
				&\leq 2t \left( L_1 \|v(x,w,t)\| + M_2 \|v(x,w,t)\| \right) \\
				&\leq 2\overline{t}\left(L_1+M_2\right) \|v(x,w,t)\|,
			\end{aligned}
		\end{displaymath}
		which indicates item (iii).
	\end{proof}
	
	Next, we show that the line-search procedure adopted in Step 3 of MPGSA is well-defined, i.e., it must terminate after a finite number of iterations.
	\begin{proposition} \label{proposition:FiniteTforlinesearch}
		Suppose that Assumption~\ref{assumption:section4:compactSetX} holds. Then there exists a constant $\overline{\alpha}\in (0,1]$ such that Step 3 of MPGSA terminates at some $\alpha_k>\overline{\alpha}\gamma$ for all $k\in\mathbb{N}$ in at most $\left\lceil \frac{\log \overline{\alpha}}{\log \gamma} \right\rceil + 1$ iterations.
	\end{proposition}
	\begin{proof}
		Thanks to $x^0\in\mathcal{X}$ and the stopping criterion in Step 3, it suffices to show the desired result by assuming $x^k\in\mathcal{X}$. Let $y^k\in \partial f(x^k)$, $z^k\in\partial h_2(x^k)$. Invoking \eqref{eq:section4:mappingV} and \eqref{eqAlgV}, it holds that $v^k = v(x^k,w^k,t^k)$. It follows from this and Lemma~\ref{lem1} (iii) that $v^k\in \mathrm{T}_{x^k}\mathcal{M} \cap \overline{\mathcal{B}}(0_n,M_3)$, where $M_3>0$ is defined in Lemma~\ref{lem1} (iii). In view of Lemma~\ref{lemma:Retr_property} (ii), the compactness of $\mathcal{X}$ and $\alpha_k v^k\in\mathrm{T}_{x^k}\mathcal{M} \cap \overline{\mathcal{B}}(0_n,M_3)$, as well as Assumption~\ref{ass1} (ii)-(iv), we deduce that there exists $L_{h_1}>0$, $L_{h_2}>0$, $L_r>0$, $L_{\nabla r}$, $L_g>0$, $L_{\nabla g}$, $L_{f}>0$ and $M_{\nabla g}$  such that for any $\alpha_k \in (0,1]$, the following inequalities hold
		\begin{align}
			\label{proof:proposition:FiniteTforlinesearch:eqnew}
			\|\nabla g(x^k)\| &\leq M_{\nabla g}, \\
			\label{proof:proposition:FiniteTforlinesearch:eq44} h_1(\tilde{x}^{k+1}) &\leq h_1(x^k+\alpha_kv^k) + L_{h_1}\|\alpha_kv^k\|^2,	\\
			\label{proof:proposition:FiniteTforlinesearch:eq45} h_2(x^k+\alpha_kv^k) &\leq h_2(\tilde{x}^{k+1}) + L_{h_2}\|\alpha_kv^k\|^2,	\\
			\label{proof:proposition:FiniteTforlinesearch:eq46} r(\tilde{x}^{k+1}) &\leq r(x^k) + \langle \nabla r(x^k),\alpha_k v^k \rangle + \frac{L_{\nabla r}+2L_r}{2}\|\alpha_kv^k\|^2,	\\
			\label{proof:proposition:FiniteTforlinesearch:eq47} |g(\tilde{x}^{k+1})-g(x^k) - \langle \nabla g(x^k),\alpha_k v^k \rangle | &\leq \frac{L_{\nabla g}+2L_g}{2} \|\alpha_kv^k\|^2,	\\
			\label{proof:proposition:FiniteTforlinesearch:eq48} f(x^k+\alpha_k v^k) &\leq f(\tilde{x}^{k+1}) + L_{f}\|\alpha_k v^k\|^2, 
		\end{align}
		where $\tilde{x}^{k+1}=\mathrm{Retr}_{x^k}(\alpha_k v^k)$. In view of Lemma~\ref{lemma:Retr_property} (i), $\alpha_k v^k\in\mathrm{T}_{x^k}\mathcal{M} \cap \overline{\mathcal{B}}(0_n,M_3)$ and $x^k\in \mathcal{X}$, we have $\tilde{x}^{k+1}\in \mathcal{X}+W_1 \overline{\mathcal{B}}(0_n,M_3)$. This together with the compactness of $\mathcal{X}$, Assumption~\ref{ass1} (iii) and (\romannumeral4) yields that there exist $\widetilde{M_f}$, $\widetilde{M_g}$, $\widetilde{L_{1/g}}$ and $\widetilde{L_f}$ that
		\begin{align}
			\label{proof:proposition:FiniteTforlinesearch:eqnew0}
			\left|f(\tilde{x}^{k+1})\right|\leq \widetilde{M_f},\  g(\tilde{x}^{k+1}) &\leq \widetilde{M_g},\, \\
			\label{proof:proposition:FiniteTforlinesearch:eqnew1}
			\left|\frac{1}{g(\tilde{x}^{k+1})} - \frac{1}{g(x^k)}\right| &\leq \widetilde{L_{1/g}}\|\alpha_k v^k\|,\\
			\label{proof:proposition:FiniteTforlinesearch:eqnew2}
			|f(\tilde{x}^{k+1}) - f(x^k)| &\leq \widetilde{L_f}\|\alpha_k v^k\|.
		\end{align}
		On the other hand, since the objective function in \eqref{eqAlgV} is strongly convex with modulus $\frac{1}{t_k}$ and $0_n\in\mathrm{T}_{x^k}\mathcal{M}$, we have from \eqref{eqAlgV} that
		\begin{equation}\label{proof:proposition:FiniteTforlinesearch:eq49}
			h_1(x^k+v^k) + \langle w^k,v^k \rangle + \frac{1}{t_k}\|v^k\|^2 \leq h_1(x^k),
		\end{equation}
		for all $w^k\in \Lambda(x^k)$.
		Also, it follows from the convexity of $h_1$ and $\alpha_k\in (0,1]$ that
		\begin{equation} \label{proof:proposition:FiniteTforlinesearch:eq410}
			h_1(x^k+\alpha_k v^k) \leq \alpha_k h_1(x^k+v^k) + (1-\alpha_k)h_1(x^k).
		\end{equation}
		Combining \eqref{proof:proposition:FiniteTforlinesearch:eq44} with \eqref{proof:proposition:FiniteTforlinesearch:eq49} and \eqref{proof:proposition:FiniteTforlinesearch:eq410}, we obtain that
		\begin{equation} \label{proof:proposition:FiniteTforlinesearch:eq411}
			h_1(\tilde{x}^{k+1}) \leq h_1(x^k) - \langle w^k,\alpha_k v^k \rangle - \left(\frac{1}{\alpha_k t_k} - L_{h_1}\right) \|\alpha_k v^k\|^2.
		\end{equation}
		The convexity of $h_2$ yields that 
		\begin{displaymath}
			h_2(x^k) + \langle z^k,\alpha_k v^k \rangle \leq h_2(x^k+\alpha_k v^k).
		\end{displaymath}
		Invoking this and \eqref{proof:proposition:FiniteTforlinesearch:eq45}, we see that
		\begin{equation} \label{proof:proposition:FiniteTforlinesearch:eq412}
			h_2(x^k) + \langle z^k,\alpha_k v^k \rangle \leq h_2(\tilde{x}^{k+1}) + L_{h_2} \|\alpha_k v^k\|^2.
		\end{equation}
		Moreover, using \eqref{proof:proposition:FiniteTforlinesearch:eqnew0}, \eqref{proof:proposition:FiniteTforlinesearch:eq47} and Lemma~\ref{lem1} (i) , we obtain further that
		\begin{equation}\label{proof:proposition:FiniteTforlinesearch:eq413}
			\begin{aligned}
				-2\frac{f(\tilde{x}^{k+1})}{g(x^k)} + \frac{f(\tilde{x}^{k+1})}{g^2(x^k)}g(\tilde{x}^{k+1}) &\leq -\frac{f(\tilde{x}^{k+1})}{g(x^k)} + \frac{f(\tilde{x}^{k+1})}{g^2(x^k)}\left(g(\tilde{x}^{k+1})-g(x^k) - \langle \nabla g(x^k),\alpha_k v^k\rangle\right) \\ &+\left\langle \frac{f(\tilde{x}^{k+1})}{g^2(x^k)}\nabla g(x^k),\alpha_k v^k \right\rangle \\ &\leq -\frac{f(\tilde{x}^{k+1})}{g(x^k)} + \left\langle \frac{f(\tilde{x}^{k+1})}{g^2(x^k)}\nabla g(x^k),\alpha_k v^k \right\rangle  \\ &+ \frac{M_1^2 \widetilde{M_f} (L_{\nabla g}+2L_g)}{2} \|\alpha_k v^k\|^2.
			\end{aligned}
		\end{equation}	
		Then we see that
		\begin{equation}\label{proof:proposition:FiniteTforlinesearch:eqnew3}
			\begin{aligned}
				\left\langle \frac{f(\tilde{x}^{k+1})}{g^2(x^k)}\nabla g(x^k),\alpha_k v^k \right\rangle &= \left\langle \frac{f(x^k)}{g^2(x^k)}\nabla g(x^k),\alpha_k v^k \right\rangle + \frac{f(\tilde{x}^{k+1}) - f(x^k)}{g^2(x^k)}\left\langle \nabla g(x^k),\alpha_k v^k \right\rangle \\ &\leq \left\langle \frac{f(x^k)}{g^2(x^k)}\nabla g(x^k),\alpha_k v^k \right\rangle + M_1^2\widetilde{L_f}M_{\nabla g}\|\alpha_k v^k\|^2,
			\end{aligned}
		\end{equation}
		where the inequality follows from \eqref{proof:proposition:FiniteTforlinesearch:eqnew}, \eqref{proof:proposition:FiniteTforlinesearch:eqnew0}, \eqref{proof:proposition:FiniteTforlinesearch:eqnew2} and Lemma~\ref{lem1} (i) . Then Assumption~\ref{ass1} (iv) leads to
		\begin{displaymath}
			\begin{aligned}
				f(x^k) + \langle y^k, \alpha_k v^k\rangle &\leq f(x^k + \alpha_k v^k) + \frac{\tau}{2}\|\alpha_k v^k\|^2\\ &\leq f(\tilde{x}^{k+1}) + \left(\frac{\tau}{2}+L_f\right)\|\alpha_k v^k\|^2,
			\end{aligned}
		\end{displaymath}
		where the last inequality follows from \eqref{proof:proposition:FiniteTforlinesearch:eq48}. Using this and Lemma~\ref{lem1} (i), we have that
		\begin{equation} \label{proof:proposition:FiniteTforlinesearch:eq414}
			-\frac{f(\tilde{x}^{k+1})}{g(x^k)} \leq -\frac{f(x^k)}{g(x^k)} - \langle \frac{y^k}{g(x^k)}, \alpha_k v^k \rangle + M_1\left(\frac{\tau}{2}+L_f\right)\|\alpha_k v^k\|^2.
		\end{equation}
		One can directly check that
		\begin{equation} \label{proof:proposition:FiniteTforlinesearch:eqnew4}
			\begin{aligned}
				-\frac{f(\tilde{x}^{k+1})}{g(\tilde{x}^{k+1})} &= \frac{f(\tilde{x}^{k+1})}{g^2(\tilde{x}^{k+1})} g(\tilde{x}^{k+1}) - 2\frac{f(\tilde{x}^{k+1})}{g(\tilde{x}^{k+1})} \\
				&= -2\frac{f(\tilde{x}^{k+1})}{g(x^k)} + \frac{f(\tilde{x}^{k+1})}{g^2(x^k)}g(\tilde{x}^{k+1}) -f(\tilde{x}^{k+1})g(\tilde{x}^{k+1})\left(\frac{1}{g(\tilde{x}^{k+1})} - \frac{1}{g(x^k)}\right)^2\\ &\leq -2\frac{f(\tilde{x}^{k+1})}{g(x^k)} + \frac{f(\tilde{x}^{k+1})}{g^2(x^k)}g(\tilde{x}^{k+1}) + \widetilde{M_f}\widetilde{M_g}\widetilde{L_{1/g}}^2\|\alpha_k v^k\|^2,
			\end{aligned}
		\end{equation}
		where the inequality follows from \eqref{proof:proposition:FiniteTforlinesearch:eqnew0} and \eqref{proof:proposition:FiniteTforlinesearch:eqnew1}. Now, by summing \eqref{proof:proposition:FiniteTforlinesearch:eq46}, \eqref{proof:proposition:FiniteTforlinesearch:eq411}-\eqref{proof:proposition:FiniteTforlinesearch:eqnew4} and rearranging terms, we obtain that
		\begin{equation} \label{proof:proposition:FiniteTforlinesearch:eq415}
			\begin{aligned}
				F(\tilde{x}^{k+1}) \leq F(x^k) - \left( \frac{1}{\alpha_k t_k}-c \right) \|\alpha_k v^k\|^2,
			\end{aligned}
		\end{equation}
		where 
		\begin{equation}  \label{eq:definition_c}
			c:= L_{h_1}+L_{h_2}+\frac{L_{\nabla r}+2L_r}{2}+\frac{M_1^2\widetilde{M_f} (L_{\nabla g}+2L_g)}{2} + M_1^2\widetilde{L_f}M_{\nabla g} + M_1L_f +\frac{\tau M_1}{2} + \widetilde{M_f}\widetilde{M_g}\widetilde{L_{1/g}}^2.
		\end{equation}
		Upon setting $\overline{\alpha} = \min \left\{ \frac{1}{2c \overline{t}},1 \right\}$, we conclude that Step 3 of MPGSA and $t_k\in [ \underline{t},\overline{t} ]$ that it must terminate at some $\alpha_k>\overline{\alpha}\gamma$. This completes the proof.
	\end{proof}
	
	From the proof of Proposition~\ref{proposition:FiniteTforlinesearch}, we deduce the following corollary immediately.
	\begin{corollary} \label{corollary:FiniteTforlinesearch}
		Suppose that Assumption~\ref{assumption:section4:compactSetX} holds. If $\overline{t}<\frac{1}{2c}$ with $c$ being defined in \eqref{eq:definition_c}, then Step 3 of MPGSA terminates at $m=0$.
	\end{corollary}
	
	Now we are ready to establish the main convergence results of MPGSA for solving problem \eqref{problemn1}.
	\begin{theorem} \label{theorem:FunctionConvergence}
		Suppose that Assumption~\ref{assumption:section4:compactSetX} holds. Let $\{x^k:k\in\mathbb{N}\}$ and $\{v^k:k\in\mathbb{N}\}$ be generated by MPGSA. Then the following statements hold.
		\begin{enumerate}[label = {\upshape(\roman*)}]
			\item $F_{*} := \lim\limits_{k\to\infty}F(x^k)$ exists and for all $k\in\mathbb{N}$,
			\begin{equation} \label{theorem:FunctionConvergence:item1}
				F(x^{k+1}) \leq F(x^k) - \frac{\overline{\alpha}\gamma}{2t_k}\|v^k\|^2,
			\end{equation}
			where $\overline{\alpha}$ is defined in Proposition~\ref{proposition:FiniteTforlinesearch}.
			\item $\lim\limits_{k\to\infty}\|v^k\| = 0$.
			\item $\{x^k:k\in\mathbb{N}\} \subseteq \mathcal{X}$ is bounded and any accumulation point of $\{x^k:k\in\mathbb{N}\}$ is a critical point of problem \eqref{problemn1}.
		\end{enumerate}
	\end{theorem}
	
	\begin{proof}
		First we prove item (i). In view of Step 3 of MPGSA and Proposition~\ref{proposition:FiniteTforlinesearch}, we see that $\{F(x^k:k\in\mathbb{N})\}$ is nonincreasing. This together with the existence of optimal solutions to problem \eqref{problemn1} indicates that $F_{*}=\lim\limits_{k\to\infty}F(x^k)$ exists. Since the line-search procedure of MPGSA terminates at some $\alpha_k>\overline{\alpha}\gamma$ for all $k\in\mathbb{N}$ from Proposition~\ref{proposition:FiniteTforlinesearch}, we immediately obtain \eqref{theorem:FunctionConvergence:item1}.
		
		Next we show item (ii). Invoking $t_k\in[\underline{t},\overline{t}]$, \eqref{theorem:FunctionConvergence:item1} yields that for all $k\in\mathbb{N}$
		\begin{displaymath}
			F(x^{k+1}) \leq F(x^k) - \frac{2\overline{\alpha}\gamma}{\overline{t}} \|v^k\|^2.
		\end{displaymath}
		Since $\lim\limits_{k\to\infty}F(x^k)$ exists from item (i), by passing to limit in the above inequality we get item (ii).
		
		Finally, we prove item (iii). The boundedness of $\{x^k:k\in\mathbb{N}\} \subseteq \mathcal{X}$ follows directly from Assumption~\ref{assumption:section4:compactSetX} and the fact that $\{F(x^k:k\in\mathbb{N})\}$ is nonincreasing. Let $x^{*}$ be an accumulation point of $\{x^k:k\in\mathbb{N}\}$ and let $\{x^{k_j}:j\in\mathbb{N}\}$ be a subsequence such that $\lim\limits_{j\to\infty}x^{k_j}=x^{*}$. In view of the first-order optimality of \eqref{eqAlgV}, there exists $\xi^{k_j}\in\partial h_1(x^{k_j}+v^{k_j})$, such that
		\begin{equation} \label{proof:theorem:FunctionConvergence:item3:starstar}
			0_n \in \xi^{k_j} + \nabla r(x^{k_j}) + \frac{f(x^{k_j}) \nabla g(x^{k_j})}{g^2(x^{k_j})} - \frac{y^{k_j}}{g(x^{k_j})} - z^{k_j} + \frac{v^{k_j}}{t_{k_j}} + \mathrm{N}_{x^{k_j}}\mathcal{M},
		\end{equation}
		where $y^{k_j} \in \partial f(x^{k_j})$ and $z^{k_j} \in \partial h_2(x^{k_j})$. Note that $\{\xi^{k_j}:k\in\mathbb{N}\}$ is bounded due to the continuity and convexity of $h_1$ as well as the boundedness of $\{x^{k_j}:j\in\mathbb{N}\}$. Similarly, we claim that $\{z^{k_j}:j\in\mathbb{N}\}$ is also bounded. Furthermore, $\{y^{k_j}:j\in\mathbb{N}\}$ is bounded thanks to the boundedness of $\{x^{k_j}:j\in\mathbb{N}\}$, the continuity and weak convexity of $f$. Hence, invoking item (ii) and by passing to a further subsequence if necessary, we may assume without loss of generality that $\lim\limits_{j\to\infty} \xi^{k_j}$, $\lim\limits_{j\to\infty}z^{k_j}$ and $\lim\limits_{j\to\infty}y^{k_j}$ all exist, and they belong to $\partial h_1(x^{*})$, $\partial h_2(x^{*})$ and $\partial f(x^{*})$, respectively. Using this and invoking $\lim\limits_{j\to\infty}v^{k_j}=0_n$ from item (ii) together with the outer semicontinuity of $\mathrm{N}_x\mathcal{M}$ with respect to $x\in\mathcal{M}$, we have upon passing to the limit in \eqref{proof:theorem:FunctionConvergence:item3:starstar} that 
		\begin{displaymath}
			0_n \in \partial h_1(x^{*}) - \partial h_2(x^{*}) + \nabla r(x^{*}) + \frac{f(x^{*})}{g^2(x^{*})}\nabla g(x^{*}) - \frac{\partial f(x^{*})}{g(x^{*})} + \mathrm{N}_{x^{*}}\mathcal{M}.
		\end{displaymath}
		In view of the above inclusion and Definition~\ref{definition:lifted_sta_And_critical}, we conclude that $x^{*}$ is a critical point of problem \eqref{problemn1}.
	\end{proof}
	
	\section{The case when $h_2$ has a special structure.}
	In this section, we consider problem \eqref{problemn1} with
	\begin{equation} \label{eq:SpecialCase:h_2}
		h_2(x) = \max \limits_{1\leq i\leq J} \psi_i(x),
	\end{equation}
	where $\psi_i : \mathbb{R}^n\to\mathbb{R}$ is convex and continuously differentiable on $\mathbb{R}^n$ for all $i=1,2,\cdots,J$. Motivated by the method for computing B-stationary point of nonsmooth DC programs over a convex set \cite{Pang-Razaviyayn-Alvarado:MATH_OPER_RES:2017}, we develop an enhanced MPGSA (EMPGSA) for solving problem \eqref{problemn1} with $h_2$ defined by \eqref{eq:SpecialCase:h_2} and show that the solution sequence generated by EMPGSA has subsequential convergence towards a lifted B-stationary point of the underlying problem. To this end, we first present another characterization of a lifted B-stationary point for problem \eqref{problemn1}. Let $\mathbb{N}_J:=\{1,2,\cdots,J\}$ and for any $x\in\mathbb{R}^n$, $\eta>0$, we define $\mathcal{A}(x):=\{i\in\mathbb{N}_J:\psi_i(x)=h_2(x)\}$ and $\mathcal{A}_\eta(x):= \{i\in\mathbb{N}_J:\psi_i(x)>h_2(x)-\eta\}$.
	
	\begin{proposition} \label{proposition41_description}
		Consider problem \eqref{problemn1} with $h_2$ defined by \eqref{eq:SpecialCase:h_2}. Then a vector $\overline{x}\in\mathcal{M}$ is a lifted B-stationary point of problem \eqref{problemn1} if and only if there exist $\overline{y}\in \partial f(\overline{x})$ such that for all $i\in\mathcal{A}(\overline{x})$,
		\begin{equation} \label{eq:proposition:Lifted_d_stationary_point:star2}
			\nabla \psi_i(\overline{x}) + \frac{\overline{y}}{g(\overline{x})} \in \partial h_1(\overline{x}) + \nabla r(\overline{x}) + \frac{f(\overline{x})\nabla g(\overline{x})}{g^2(\overline{x})} + \mathrm{N}_{\overline{x}}\mathcal{M},
		\end{equation}
		or equivalently, for some $t\in(0,+\infty)$,
		\begin{equation} \label{eq:proposition:Lifted_d_stationary_point:star3}
			\begin{aligned}
				0_n \in \mathop{\arg\min} \limits_{v\in\mathrm{T}_{\overline{x}} \mathcal{M}} \left\{ h_1(\overline{x}+v) + \left\langle \nabla r(\overline{x}) + \frac{f(\overline{x}) \nabla g(\overline{x})}{g^2(\overline{x})} - \frac{\overline{y}}{g(\overline{x})} - \nabla \psi_i(\overline{x})  , v \right\rangle  \right.\\
				\left. +\frac{1}{2t} \|v\|^2 \right\}.
			\end{aligned}	
		\end{equation}
	\end{proposition}
	\begin{proof}
		Since $\partial h_2(\overline{x})$ is the convex hull of $\{\nabla \psi_i(\overline{x}):i\in\mathcal{A}(\overline{x})\}$, we deduce \eqref{eq:proposition:Lifted_d_stationary_point:star2} immediately from Definition~\ref{definition:lifted_sta_And_critical} (i) and the convexity of $\partial h_1(\overline{x})$ and $\mathrm{N}_{\overline{x}}\mathcal{M}$. Furthermore, \eqref{eq:proposition:Lifted_d_stationary_point:star3} is equivalent to \eqref{eq:proposition:Lifted_d_stationary_point:star2} due to the convexity of the optimization problem involved in \eqref{eq:proposition:Lifted_d_stationary_point:star3} and the Fermat's rule. This completes the proof.
	\end{proof}
	
	Now we propose a new algorithm for solving \eqref{problemn1} with $h_2$ defined by \eqref{eq:SpecialCase:h_2}, which is extension of MPGSA and presented as Algorithm~\ref{algorithm2} below. Note that Step 2 and Step 3 (c) are motivated by the work of Pang, Razaviyayn and Alvarado \cite{Pang-Razaviyayn-Alvarado:MATH_OPER_RES:2017}, which develop an enhanced DC algorithm that converges to a B-stationary point of the associated DC programs. Similar to the enhanced DC program of \cite{Pang-Razaviyayn-Alvarado:MATH_OPER_RES:2017}, we need to solve the subproblem \eqref{eq:subproblem_MPGSA_d} for $|\mathcal{A}_\eta(x^k)|$ time in Step 2 of EMPGSA. Although this may increase the computational cost of Step 2, as we will see later, the algorithm can be guaranteed to converge to a lifted B-stationary point of problem \eqref{problemn1}. Next we conduct convergence analysis for EMPGSA.
	
	
	\begin{algorithm}[H]
		\caption{\ ${\rm EMPGSA}$ for solving \eqref{problemn1} with $h_2$ defined by \eqref{eq:SpecialCase:h_2}}\label{algorithm2}
		\begin{algorithmic}
			\State \vspace{-0.15 cm}
			\begin{description}
				\item[\bf Step 0.] Input $x^0 \in \mathcal{M}$, $\gamma \in (0,1)$, $\eta > 0$, $t_k\in (\underline{t},\overline{t})$ with $0<\underline{t}<\overline{t}$, for $k\in \mathbb{N}$. Set $k \leftarrow 0$.\vspace{0.1 cm}
				\item[\bf Step 1.] Choose $y^k \in \partial f(x^k)$.\vspace{0.1 cm}
				\item[\bf Step 2.] For each $i\in\mathcal{A}_\eta(x^k)$, compute
				\begin{equation}\label{eq:subproblem_MPGSA_d}
					\begin{aligned}
						v^{k,i} := \mathop{\arg\min}\limits_{v \in \mathrm{T}_{x^k}\mathcal{M}} \left\{ h_1(x^k+v) + \Bigg\langle \nabla r(x^k) + \frac{f(x^k) \nabla g(x^k)}{g^2(x^k)} - \frac{y^k}{g(x^k)} \right.\\
						\left.- \nabla \psi_i(x^k), v \Bigg\rangle+ \frac{1}{2 t_k} \|v\|^2 \right\}.
					\end{aligned}
				\end{equation}
				\vspace{0.1 cm}
				\item[\bf Step 3.] For $m=0,1,\dots,$ do
				\begin{itemize}
					\item[(a)] $\alpha_k = \gamma^m$,
					\item[(b)] $x^{k,i} = \mathrm{Retr}_{x^k}(\alpha_k v^{k,i})$, for all $i\in \mathcal{A}_\eta(x^k)$,
					\item[(c)] Let $\hat{i} \in \mathop{\arg\min} \left\{ F(x^{k,i}) + \frac{\alpha_k}{4 t_k} \|v^{k,i}\|^2 : i \in \mathcal{A}_\eta(x^k) \right\}$,
					\item[(d)] If $x^{k,\hat{i}}$ satisfies
					\[
					F(x^{k,\hat{i}}) \leq \left( h_1-\psi_i+r-\frac{f}{g} \right)(x^k) - \frac{\alpha_k}{4 t_k} \|v^{k,\hat{i}}\|^2 
					- \frac{\alpha_k}{4 t_k} \|v^{k,i}\|^2, \quad \forall i\in\mathcal{A}_\eta(x^k),
					\]
					set $x^{k+1} = x^{k,\hat{i}}$ and go to Step 4.
				\end{itemize}
				\vspace{0.1 cm}
				\item[\bf Step 4.] Set $k \leftarrow k+1$ and go to Step 1.
			\end{description}
		\end{algorithmic}
	\end{algorithm}
	
	In view of the finiteness of $J$, we directly obtain an auxiliary lemma below by following the same line of arguments in Lemma~\ref{lem1} and Proposition~\ref{proposition:FiniteTforlinesearch}. The proof are essentially the same and thus omitted here for brevity.
	
	\begin{lemma} \label{lemma:prework_forProving_LS_welldefined_specialCase}
		Suppose that Assumption~\ref{assumption:section4:compactSetX} holds and $x^k\in\mathcal{X}$. Let $v^{k,i}$ be defined by \eqref{eq:subproblem_MPGSA_d} and $x^{k,i}:=\mathrm{Retr}_{x^k}(\alpha v^{k,i})$ for all $i\in\mathbb{N}_J$ and some $\alpha\in (0,1]$. Then there exists $\tilde{\alpha}\in (0,1]$ that is independent of $x^k$, such that for all $i\in\mathbb{N}_J$ and $\alpha\in (0,\tilde{\alpha}]$, it holds that
		\begin{displaymath}
			\left( h_1 - \psi_i + r - \frac{f}{g} \right)(x^{k,j}) \leq \left( h_1 - \psi_i + r - \frac{f}{g} \right)(x^k) - \frac{\alpha}{2 t_k} \|v^{k,j}\|^2.
		\end{displaymath}
	\end{lemma}
	
	Equipped with Lemma~\ref{lemma:prework_forProving_LS_welldefined_specialCase}, we next prove that the line-search procedure in Step 3 of EMPGSA is well-defined.
	
	\begin{proposition} \label{proposition:finiteIteration_specialCase}
		Suppose that Assumption~\ref{assumption:section4:compactSetX} holds and let $\tilde{\alpha}$ be defined by Lemma~\ref{lemma:prework_forProving_LS_welldefined_specialCase}. Then Step 3 of EMPGSA terminates at some $\alpha_k>\tilde{\alpha}\gamma$ for all $k\in\mathbb{N}$ in at most $\left\lceil \frac{\log \tilde{\alpha}}{\log \gamma} \right\rceil + 1$ iterations.
	\end{proposition}
	\begin{proof}
		Since $x_0\in\mathcal{M}$ and the stopping rule in Step 3 implies that $\{F(x^k):k\in\mathbb{N}\}$ is nonincreasing, it suffices to prove this proposition by assuming $x^k\in\mathcal{X}$. Hence, invoking the choice of $\hat{i}$ in Step 3 (c) of EMPGSA, we have for any $i \in \mathcal{A}_\eta(x^k)$ and $\alpha_k \in (0,\tilde{\alpha}]$ that 
		\begin{displaymath}
			\begin{aligned}
				F(x^{k,\hat{i}}) &\leq F(x^{k,i}) + \frac{\alpha_k}{4 t_k} \|v^{k,i}\|^2 - \frac{\alpha_k}{4 t_k} \|v^{k,\hat{i}}\|^2 \\
				&\leq \left( h_1-\psi_i+r-\frac{f}{g} \right)(x^{k,i}) + \frac{\alpha_k}{4 t_k} \|v^{k,i}\|^2 - \frac{\alpha_k}{4 t_k} \|v^{k,\hat{i}}\|^2 		\\
				&\leq \left( h_1-\psi_i+r-\frac{f}{g} \right)(x^k) - \frac{\alpha_k}{4 t_k}\|v^{k,i}\|^2 - \frac{\alpha_k}{4 t_k}\|v^{k,\hat{i}}\|^2,
			\end{aligned}
		\end{displaymath}
		where the second inequality follows from the fact that $h_2(x^{k,i})\geq\psi_i(x^{k,i})$, $\forall i\in\mathcal{A}_\eta(x^{k,i})$, and the last inequality holds thanks to Lemma~\ref{lemma:prework_forProving_LS_welldefined_specialCase}. This together with the updating rule of $\alpha_k$ indicates the desired result.
	\end{proof}

	Now we establish the convergence results of EMPGSA in the following theorem.
	\begin{theorem} \label{theorem:convergence_result_algorithm5_1}
		Suppose that Assumption~\ref{assumption:section4:compactSetX} holds. Let $\{x^k:k\in\mathbb{N}\}$ be the solution sequence generated by EMPGSA. Then the following statements hold.
		\begin{enumerate}[label = {\upshape(\roman*)}]
			\item For all $k\in\mathbb{N}$, it holds that
			\begin{equation}  \label{theorem:convergence_result_algorithm5_1:item1}
				F(x^{k+1}) \leq F(x^k) - \frac{\tilde{\alpha}\gamma}{4 t_k} \|v^{k,i}\|^2,~\forall i\in\mathcal{A}_\eta(x^k). \\
			\end{equation}
			Consequently, $F_{*}=\lim\limits_{k\to\infty} F(x^k)$ exists.
			\item $\{x^k:k\in\mathbb{N}\}$ is bounded and any accumulation point of $\{x^k:k\in\mathbb{N}\}$ is a lifted B-stationary point of problem \eqref{problemn1}.
		\end{enumerate}
	\end{theorem}
	\begin{proof}
		First we prove item (i). In view of Proposition~\ref{proposition:finiteIteration_specialCase} and the fact that $h_2(x^k)=\max \{\psi_i(x^k):i\in\mathcal{A}_\eta(x^k)\}$, we can directly deduce that \eqref{theorem:convergence_result_algorithm5_1:item1} holds. Hence, $\{F(x^k):k\in\mathbb{N}\}$ is nonincreasing. This together with the existence of optimal solutions to problem \eqref{problemn1} implies that $F_{*}=\lim\limits_{k\to\infty}F(x^k)$ exists.
		
		Next we show item (ii). The boundedness of $\{x^k\in\mathbb{N}\}$ follows immediately from Assumption~\ref{assumption:section4:compactSetX} and the fact that $\{F(x^k):k\in\mathbb{N}\}$ is nonincreasing. Let $x^{*}$ be an accumulation point $\{x^k:k\in\mathbb{N}\}$ and let $\{x^{k_j}:j\in\mathbb{N}\}$ be a subsequence such that $\lim\limits_{j\to\infty}x^{k_j}=x^{*}$. Pick an arbitrary $i\in\mathcal{A}(x^{*})$. By the definition of $\mathcal{A}^{*}$ and $\mathcal{A}_\eta(x^{k_j})$, $\lim\limits_{j\to\infty}x^{k_j}=x^{*}$, and the continuity of $\psi_i$, it is not difficult to see that there exists $j_0\in\mathbb{N}$ such that 
		\begin{equation}  \label{proof:theorem:convergence_result_algorithm5_1:star2}
			i\in \mathcal{A}_\eta(x^{k_j}),~ \mbox{for all }j>j_0.
		\end{equation}
		Using this, $t_k\in [\underline{t},\overline{t}]$ and \eqref{theorem:convergence_result_algorithm5_1:item1}, we further obtain that
		\begin{displaymath}
			F(x^{k_j+1}) \leq F(x^{k_j}) - \frac{\tilde{\alpha}\gamma}{4 \overline{t}}\|v^{k_j,i}\|^2,~\mbox{for }j>j_0.
		\end{displaymath}
		By passing to the limit in the above inequality, we have that
		\begin{equation}  \label{proof:theorem:convergence_result_algorithm5_1:star3}
			\lim \limits_{j\to\infty}\|v^{k_j,i}\|=0.
		\end{equation}
		Moreover, in view of the first-order optimality of subproblem \eqref{eq:subproblem_MPGSA_d}, for all $j>j_0$ there exists $\xi^{k_j}\in \partial h_1(x^{k_j}+v^{k_j,i})$, such that
		\begin{equation}  \label{proof:theorem:convergence_result_algorithm5_1:star4}
			0_n\in \xi^{k_j} + \nabla r(x^{k_j}) + \frac{f(x^{k_j})\nabla g(x^{k_j})}{g^2(x^{k_j})} - \frac{y^{k_j}}{g(x^{k_j})} - \nabla \psi_i(x^{k_j}) + \frac{v^{k_j}}{t_{k_j}} + \mathrm{N}_{x^{k_j}}\mathcal{M},
		\end{equation}
		where $y^{k_j}\in \partial f(x^{k_j})$. Now, by following the same argument as those in the proof of Theorem~\ref{theorem:FunctionConvergence} (iii), we deduce from the above display that 
		\begin{displaymath}
			0_n \in \partial h_1(x^{*}) + \nabla r(x^{*}) + \frac{f(x^{*})\nabla g(x^{*})}{g^2(x^{*})} - \frac{\partial f(x^{*})}{g(x^{*})} - \nabla \psi_i(x^{*}) + \mathrm{N}_{x^{*}}\mathcal{M}.
		\end{displaymath}
		Since the index $i$ is chosen arbitrarily from $\mathcal{A}(x^{*})$, we conclude from the above inclusion and Proposition~\ref{proposition41_description} that $x^{*}$ is a lifted B-stationary point of problem \eqref{problemn1}. This completes the proof.
	\end{proof}
	
	\section{Sequential convergence of MPGSA.}
	In this section, we shall establish the convergence of the whole solution sequence $\{x^k:k\in\mathbb{N}\}$ generated by MPGSA under further assumption. Since the ManPG method is a special case of MPGSA as mentioned in section 4, our convergence results also imply the convergence of the entire solution sequence generated by the ManPG, which seems unknown in the literature. In order to facilitate our sequential convergence analysis, we first define several quantities in the following Lemma.
	\begin{lemma}\label{lemsec6}
		The following quantities are positive and finite:
		\begin{enumerate}[label = {\upshape(\roman*)}]
			\item $M_f:=\sup\limits_{x\in\mathcal{X}}|f(x)|$, $\ M_g:=\sup\limits_{x\in\mathcal{X}}g(x)$, $\ m_g:=\inf\limits_{x\in\mathcal{X}}g(x)$, $\ M_y:=\sup\{\|y\|:y\in\partial f(x), x\in\mathcal{X}\}$;
			\item $M_g':=\sup\{g(x): x\in\mathcal{X}+\overline{\mathcal{B}}(0_n, M_3)\}$;
			\item $L_{1/g}$ is the Lipschitz modulus of $1/g$ on $\mathcal{X}$, $L_g'$ is the Lipschitz modulus of $g$ on $\mathcal{X}+\overline{\mathcal{B}}(0_n, M_3)$.
		\end{enumerate}
	\end{lemma}
	
	Then we introduce a useful lemma concerning MPGSA.
	\begin{lemma} \label{lemma:Chapter6:lem62}
		Suppose Assumption~\ref{assumption:section4:compactSetX} holds. Let $\{(x^k,y^k,z^k,v^k):k\in\mathbb{N}\}$ be generated by MPGSA. Then, there exists a compact set $\Omega \subseteq \mathbb{R}^d$ $(\Omega = \{0\}$ if $d=0)$ such that for any $k\in\mathbb{N}$ and any local defining function\footnote{If $\mathcal{M}$ is an open submanifold of $\mathbb{R}^n$, i.e., $d=0$, we define $\Phi_k=0$ for all $k\in\mathbb{N}$.} $\Phi_k:\mathbb{R}^n\to\mathbb{R}^d$ of $\mathcal{M}$ at $x^k$, there exists $s^k\in\Omega$ satisfying
		\begin{equation}	\label{eq:lemma61}
			0_n \in \partial h_1(x^k+v^k) - z^k + \nabla r(x^k) + \frac{f(x^k) \nabla g(x^k)}{g^2(x^k)} - \frac{y^k}{g(x^k)} + \frac{1}{t_k} v^k + \mathrm{D} \Phi_k(x^k)^{*}(s^k).
		\end{equation}
	\end{lemma}
	\begin{proof}
		If $\mathcal{M}$ is an open submanifold of $\mathbb{R}^n$, \cite[Theorem 3.15]{Boumal:Intro_Opt_Man:2023} yields that $\mathrm{T}_{x^k}\mathcal{M}=\mathbb{R}^n$ for $k\in\mathbb{N}$. Then, we obtain \eqref{eq:lemma61} with $s_k=0$ by applying the first-order optimality condition for the convex problem \eqref{eqAlgV}. The compactness of $\Omega$ is obvious since $\Omega$ can be chosen as $\{0\}$. 
		
		We next consider the case when $\mathcal{M}$ is not an open submanifold of $\mathbb{R}^n$. Let $\Phi_k:\mathbb{R}^n\to\mathbb{R}^d$ be any local defining function for $\mathcal{M}$ at $x^k$. According to \cite[Theorem 3.15]{Boumal:Intro_Opt_Man:2023} and \cite[Example 6.8]{Rockafellar-Wets:VariationalAnalysis:2009}, we have $\mathrm{T}_{x^k}\mathcal{M}=\mathrm{ker}~\mathrm{D}\Phi_k(x^k)$ and $\mathrm{N}_{x^k}\mathcal{M}=\{\mathrm{D}\Phi_k(x^k)^{*}(s):s\in\mathbb{R}^d\}$. Invoking \cite[Theorem 8.15]{Rockafellar-Wets:VariationalAnalysis:2009}, we deduce \eqref{eq:lemma61} from the convexity of problem \eqref{eqAlgV}, since its objective function is Lipschitz continuous. We then show the boundedness of $\{s^k:k\in\mathbb{N}\}$. Theorem~\ref{theorem:FunctionConvergence} indicates that $\{(x^k,v^k):k\in\mathbb{N}\}$ is bounded. This implies that $\bigcup_{k\in\mathbb{N}}\partial h_1(x^k+v^k)$, $\bigcup_{k\in\mathbb{N}}\partial h_2(x^k)$, $\bigcup_{k\in\mathbb{N}}\partial f(x^k)$ and $\left\{\left(\nabla r(x^k),\nabla g(x^k)\right):k\in\mathbb{N}\right\}$ are bounded, since $h_1$, $h_2$ are convex, $f$ is weakly convex and $\nabla r$, $\nabla g$ are locally Lipschitz continuous. In addition, $g(x^k)\geq m_g$ for all $k\in\mathbb{N}$ with $m_g$ defined by Lemma~\ref{lemsec6} (i). Therefore, these together with \eqref{eq:lemma61} lead to the boundedness of $\{\mathrm{D}\Phi_k(x^k)^{*}(s^k):k\in\mathbb{N}\}$. Finally, we show $\{s^k:k\in\mathbb{N}\}$ is bounded by contradiction. Suppose $\{s^k:k\in\mathbb{N}\}$ is not bounded. Then, we obtain from the boundedness of $\{x^k:k\in\mathbb{N}\}$ that there exist subsequence $s^{k_j}$ and $x^{k_j}$ such that
		\begin{equation} \label{proof:lemma61:eq1}
			\lim\limits_{j\to\infty} \|s^{k_j}\| = \infty \quad \mbox{and} \quad \lim\limits_{j\to\infty} x^{k_j}=x^{*},
		\end{equation}
		where $x^{*}\in\mathcal{M}$ due to Theorem~\ref{theorem:FunctionConvergence} (iii). In view of Definition~\ref{definition:embeddedSubmanifold}, there exist $J>0$ and $\Phi_{*}:\mathbb{R}^n\to\mathbb{R}^d$ such that $\Phi_{*}$ is a local defining function for $\mathcal{M}$ at $x^{*}$ and $x^{k_j}$ for all $j>J$. Then, \eqref{proof:lemma61:eq1} together with the fact that $\mathrm{rank}~\mathrm{D}\Phi_{*}(x^{k_j})=\mathrm{rank}~\mathrm{D}\Phi_{*}(x^{*})=d$ for all $j>J$ yields that 
		\begin{displaymath}
			\lim\limits_{j\to\infty} \|\mathrm{D}\Phi_{*}(x^{k_j})^{*}(s^{k_j})\| = \infty
		\end{displaymath}
		due to the smoothness of $\mathrm{D}\Phi_{*}$. This contradicts to that $\{ \mathrm{D}\Phi_k(x^k)^{*}(s^k): k\in\mathbb{N} \}$ is bounded for $\Phi_k$ being any local defining function for $\mathcal{M}$ at $x^k$. We then complete the proof immediately.
	\end{proof}

	Our sequential convergence analysis for MPGSA necessitates the following assumption on the function $f$ of problem \eqref{problemn1}. 
	\begin{assumption} \label{assumption:Chapter6}
		The function $f$ is convex on $\mathbb{R}^n$ and $f^{*}$ is continuous on $\mathrm{dom}\left( f^{*}\right)$.
	\end{assumption}
	
	Under Assumption~\ref{assumption:Chapter6}, we further introduce an auxiliary function $H:\mathbb{R}^n\times\mathbb{R}^n\times\mathbb{R}^n\times\mathbb{R}^n\times\mathbb{R}\to(-\infty,+\infty]$, defined at ($x$, $y$, $z$, $v$, $c$) as
	\begin{equation} \label{eq:definition:H}
		H(x,y,z,v,c) := Q(x+v,y,z,c) + \iota_{\mathrm{T}\mathcal{M}}(x,v) + \tilde{L}\|v\|^2,
	\end{equation}
	where $\tilde{L} := \overline{L} + \beta$ with $\beta$ will be defined in Lemma~\ref{lemma:Q_property} (ii) and
	\begin{equation} \label{eq:defenition:tildeL}
		\overline{L} := L_{h_1} + L_{h_2} + L_r + 2M_1L_{f} + M_1^2(M_fL_g + M_g'L_f)
	\end{equation}
	with $L_{h_1}$, $L_{h_2}$, $L_r$, $L_{f}$, $L_g$, $M_1$, $M_f$ and $M_g'$ defined by Proposition~\ref{proposition:FiniteTforlinesearch}, Lemma~\ref{lem1} (i) and Lemma~\ref{lemsec6} (ii), $\iota_{\mathrm{T}\mathcal{M}}$ is the indicator function on $\mathrm{T}\mathcal{M}$, that is,
	\begin{displaymath}
		\iota_{\mathrm{T}\mathcal{M}}(x,v)= \begin{cases}
			0 &,\quad \text{if }(x,v)\in\mathrm{T} \mathcal{M},\\
			+\infty&,\quad \text{otherwise},
		\end{cases}
	\end{displaymath}
	and $Q:\mathbb{R}^n\times\mathbb{R}^n\times\mathbb{R}^n\times\mathbb{R}\to (-\infty,+\infty]$ is defined at ($x$, $y$, $z$, $c$) as $Q(x,y,z,c)=+\infty$ if $y \notin \mathrm{dom} \left(f^{*}\right)$, otherwise,
	\begin{displaymath}
		Q(x,y,z,c) := (h_1+r)(x) + h_2^{*}(z) - \langle x,z \rangle + \Big( \langle x,y \rangle - f^*(y) \Big)\left(c^2g(x) - 2c\right).
	\end{displaymath}
	
	Next, we present several properties regarding the function $Q$ and the sequence $\{\left(x^k,y^k,z^k\right):k\in\mathbb{N}\}$ generated by MPGSA in the next lemma, which play important roles in our analysis. 
	
	\begin{lemma}\label{lemma:Q_property}
		Suppose that Assumption~\ref{assumption:section4:compactSetX} and Assumption~\ref{assumption:Chapter6} hold. Let $\{(x^k,y^k,z^k,v^k):k\in \mathbb{N}\}$ be generated by MPGSA. Let $c_k := 1/g(x^k)$ for $k\in\mathbb{N}$.Then the following statements hold for all $k\in \mathbb{N}$:
		\begin{enumerate}[label = {\upshape(\roman*)}]
			\item  $Q(x^k,y^k,z^k,c_k) = F(x^k)$;
			\item There exists a positive constant $\beta \geq 0$ such that
			\begin{equation}\label{lem7:eq1}
				Q(x^k, y^k, z^k, c_k) \leq Q(x^k, y^k, z^k, c_{k-1}) + \beta\|v^{k-1}\|^2.
			\end{equation}
			Furthermore, if for some $K>0$ such that $f(x^k)\geq 0$ for all $k\geq K$, then \eqref{lem7:eq1} holds for all $k\geq K$ with $\beta=0$.
		\end{enumerate}
	\end{lemma}
	\begin{proof}
		Under Assumption~\ref{assumption:Chapter6},  it follows from \cite[Proposition 16.10]{Bauschke-Combettes:HilbertSpaces:2017}, $y^k \in \partial f(x^k) $ and $z^k \in \partial h_2(x^k)$ that $f(x^k) + f^*(y^k)=\langle x^k,y^k\rangle$ as well as $h_2(x^k) + h_2^*(z^k)=\langle x^k,z^k\rangle$. These directly imply item (i) due to $c_k=1/g(x^k)$. We next prove Item (ii). By simple calculations, we have 
		\begin{equation*}
			\begin{aligned}
				|Q(x^k, y^k, z^k, c_k) - Q(x^k, y^k, z^k, c_{k-1})| &= |-f(x^k)\left(\left(2c_k-c_k^2g(x^k)\right) - \left(2c_{k-1}-c_{k-1}^2g(x^k)\right) \right)|\\ &= |-f(x^k)g(x^k)\left( \frac{1}{g(x^k)} - \frac{1}{g(x^{k-1})} \right)^2| \\ &\leq M_fM_gL_{1/g}^2\|x^k-x^{k-1}\|^2 \\ &\leq M_fM_gL_{1/g}^2W_1^2\|v^k\|^2,
			\end{aligned}
		\end{equation*}
		where the last inequality comes from $x^k = \mathrm{Retr}_{x^{k-1}}(\alpha_{k-1}v^{k-1})$, Lemma~\ref{lemma:Retr_property} (i) as well as $\alpha_{k-1}\leq1$, $M_f$, $M_g$ and $L_{1/g}$ are defined by Lemma~\ref{lemsec6}. This indicates \eqref{lem7:eq1} with $\beta:=M_fM_gL_{1/g}^2 W_1^2$. If $f(x^k) \geq 0$, it follows from $c_k \in \arg\min\{f(x^k)\left(g(x^k)c^2-2c\right):c\in\mathbb{R}\}$ that $Q(x^k, y^k, z^k, c_k) \leq Q(x^k, y^k, z^k, c_{k-1})$, which implies $\beta =0$. We then complete the proof. 
	\end{proof}
	
	With the help of the above properties, we obtain the following key proposition.
	
	\begin{proposition} \label{Case1:Prop1}
		Suppose that Assumption~\ref{assumption:section4:compactSetX} and Assumption~\ref{assumption:Chapter6} hold. Let $\{(x^k,y^k,z^k,v^k):k\in \mathbb{N}\}$ be generated by MPGSA. Let  $c_k:=1/g(x^k)$ for $k\in\mathbb{N}$. Let the parameter $\overline{t}<\min\left\{\frac{1}{2c}, \frac{1}{2\tilde{c}}, \frac{m_g}{2L_g'(L_1+M_2)}\right\}$, where $c$ is defined in \eqref{eq:definition_c} and
		\begin{align}\label{prop8:eq0}
			\tilde{c}:= L_{h_1}+L_{h_2}+\frac{L_{\nabla r}+2L_r}{2}+\frac{M_1^2M_f (L_{\nabla g}+2L_g)}{2} +M_1^2M_g'L_f + 2M_1L_f + M_1^2M_yL_g' + \beta
		\end{align}
		with $M_f$, $m_g$, $M_g'$, $M_y$ and $L_g'$ being defined in Lemma~\ref{lemsec6}, $L_{h_1}$, $L_{h_2}$, $L_{\nabla r}$, $L_r$, $L_f$ and $L_g$ being defined in the proof of Proposition~\ref{proposition:FiniteTforlinesearch}, $L_1$ being defined in Assumption~\ref{ass1} (ii), $\beta$ being defined in Lemma~\ref{lemma:Q_property} (ii), and $M_1$, $M_2$ being defined in Lemma~\ref{lem1}. Then the following statements hold:
		\begin{enumerate}[label = {\upshape(\roman*)}]
			\item There exists $a>0$ such that
			\begin{displaymath}
				H(x^k,y^k,z^k,v^k,c_k) \leq H(x^{k-1},y^{k-1},z^{k-1},v^{k-1},c_{k-1}) - a \|v^k\|^2;
			\end{displaymath}
			\item There exists $b>0$ such that
			\begin{displaymath}
				\mathrm{dist} \left( 0_{4n +1},\partial H(x^k,y^k,z^k,v^k,c_k) \right) \leq b \|v^k\|.
			\end{displaymath}
		\end{enumerate}
	\end{proposition}
	\begin{proof}
		We first prove Item (i). Corollary~\ref{corollary:FiniteTforlinesearch} indicates that $\alpha_k=1$ and $x^{k+1}=\mathrm{Retr}_{x^k}(v^k)$ for $k\in\mathbb{N}$. Thanks to Lemma~\ref{lemma:Retr_property} (ii), there exist $L_{\nabla r}$, $L_{\nabla g}>0$ such that
		\begin{align}
			\label{prop8:eq1}
			r(x^k+v^k) &\leq r(x^k) + \langle \nabla r(x^k),v^k \rangle + \frac{L_{\nabla r}}{2} \|v^k\|^2, \\
			\label{prop8:eq2}
			\frac{ L_{\nabla g}}{2} \|v^k\|^2 &\geq \left|g(x^k+v^k) - g(x^k) - \langle \nabla g(x^k),v^k \rangle\right|.
		\end{align}
		Then we obtain from $y^k \in \partial f(x^k)$ and $c_k = 1/g(x^k)$ that 
		\begin{equation}\label{prop8:eq3}
			\begin{aligned}
				\left( f^*(y^k)-\langle x^k +v^k, y^k \rangle \right)\left(2c_k-c_k^2g(x^k+v^k)\right)&=\left( f^*(y^k)-\langle x^k, y^k \rangle\right)\left(2c_k-c_k^2g(x^k)\right)-\langle v^k, \frac{y^k}{g(x^k)}\rangle\\&-c_k^2\left(f(x^k)+\langle v^k, y^k\rangle\right)\left(g(x^k)-g(x^k+v^k)\right).
			\end{aligned}
		\end{equation}
		It follows from \eqref{prop8:eq2} that
		\begin{equation}\label{prop:eq4}
			\begin{aligned}
				-c_k^2f(x^k)\left(g(x^k)-g(x^k+v^k)\right)&=c_k^2f(x^k)\left(g(x^k+v^k)-g(x^k)-\langle \nabla g(x^k), v^k\rangle\right) \\&+ \frac{f(x^k)}{g^2(x^k)}\langle\nabla g(x^k), v^k\rangle\\ &\leq \frac{f(x^k)}{g^2(x^k)}\langle\nabla g(x^k), v^k\rangle+\frac{M_1^2M_fL_{\nabla g}}{2}\|v^k\|^2.
			\end{aligned}
		\end{equation}
		Moreover, we can directly obtain that
		\begin{equation}\label{prop8:eq5}
			\begin{aligned}
				-c_k^2\langle v^k, y^k\rangle\left(g(x^k)-g(x^k+v^k)\right) 	&\leq c_k^2\left|g(x^k + v^k)-g(x^k)\right|\|y^k\|\|v^k\|\\ &\leq M_1^2M_yL_g'\|v^k\|^2.
			\end{aligned}
		\end{equation}
		Combining \eqref{proof:proposition:FiniteTforlinesearch:eq49}, \eqref{prop8:eq1} and \eqref{prop8:eq3}-\eqref{prop8:eq5}, we obtain from $z^k\in\partial h_2(x^k)$ that
		\begin{equation}	\label{proof:lem52-item1:neq4}
			\begin{aligned}
				Q(x^k+v^k,y^k,z^k,c_k) &\leq (h_1+r)(x^k)+\left(c_k^2g(x^k) - 2c_k\right)\left( \langle x^k,y^k \rangle - f^{*}(y^k) \right)\\
				& + h_2^{*}(z^k) - \langle x^k,z^k \rangle  - a_k \|v^k\|^2 \\
				&= Q(x^k,y^k,z^k,c_k) - a_k \|v^k\|^2,
			\end{aligned}
		\end{equation}
		where $a_k:=\frac{1}{t_k}-\frac{L_{\nabla r}}{2}-M_1^2M_yL_g'-\frac{M_1^2M_fL_{\nabla g}}{2}$. Noting that $\overline{t}< \frac{m_g}{2L_g'(L_1+M_2)}$, we obtain from Lemma~\ref{lem1} (iii) and Theorem~\ref{theorem:FunctionConvergence} (iii) that
		\begin{equation}\label{prop8:eq7}
			\begin{aligned}
				2c_k-c_k^2g(x^k+v^k) &= \frac{g(x^k)+\left(g(x^k)-g(x^k+v^k)\right)}{g^2(x^k)}\\&\geq \frac{m_g-L_g'\|v^k\|}{g^2(x^k)}\\&\geq \frac{m_g-2\overline{t}L_g'(L_1+M_2)}{g^2(x^k)}>0.
			\end{aligned}
		\end{equation}
		This together with the Fenchel-Young inequality \cite[Proposition 13.15]{Bauschke-Combettes:HilbertSpaces:2017} leads to 
		\begin{align}\label{prop8:eq8}
			(h+r+c_k^2fg - 2c_k f)(x^k + v^k) \leq  Q(x^k +v^k, y^k, z^k, c_k).
		\end{align}
		It follows from Lemma~\ref{lemma:Retr_property} (ii) and $x^k=\mathrm{Retr}_{x^{k-1}}(v^{k-1})$ that
		\begin{equation}\label{prop8:eq6}
			\begin{aligned}
				f(x^{k})g(x^{k}) - f(x^{k-1} + v^{k-1})g(x^{k-1} + v^{k-1}) &= 	f(x^{k})\left(g(x^{k})-g(x^{k-1} + v^{k-1})\right)\\ &+g(x^{k-1} + v^{k-1})\left(f(x^{k})-f(x^{k-1} + v^{k-1})\right) \\&\leq \left(M_fL_g+M_g'L_f\right)\|v^{k-1}\|^2.
			\end{aligned} 
		\end{equation}
		Invoking $y^k\in\partial f(x^k)$ and $z^k \in \partial h_2(x^k)$, we obtain from Lemma~\ref{lemma:Retr_property} (ii), \eqref{prop8:eq8} and \eqref{prop8:eq6} that
		\begin{equation} \label{proof:lem52-item1:neq5}
			\begin{aligned}
				Q(x^k,y^k,z^k,c_{k-1}) &= (h+r+c_{k-1}^2fg-2c_{k-1}f)(x^k) \\
				&\leq (h+r+c_{k-1}^2fg-2c_{k-1}f)(x^{k-1}+v^{k-1}) + \overline{L} \|v^{k-1}\|^2 \\
				&\leq Q(x^{k-1}+v^{k-1},y^{k-1},z^{k-1},c_{k-1}) + \overline{L}\|v^{k-1}\|^2,
			\end{aligned}
		\end{equation}
		where $\overline{L}$ is defined in \eqref{eq:defenition:tildeL}. Lemma~\ref{lemma:Q_property} (ii), \eqref{proof:lem52-item1:neq4} and \eqref{proof:lem52-item1:neq5} lead to 
		\begin{equation} \label{proof:lem52-item1:neq6}
			\begin{aligned}
				Q(x^k+v^k,y^k,z^k,c_k) &\leq Q(x^{k-1}+v^{k-1},y^{k-1},z^{k-1},c_{k-1}) + (\overline{L} + \beta)\|v^{k-1}\|^2 - a_k\|v^k\|^2 \\ &= Q(x^{k-1}+v^{k-1},y^{k-1},z^{k-1},c_{k-1}) + \tilde{L}\|v^{k-1}\|^2 - a_k\|v^k\|^2.
			\end{aligned}
		\end{equation}
		In view of $x^k \in \mathcal{M}$ and $v^k \in \mathrm{T}_{x^k}\mathcal{M}$ for all $k\in\mathbb{N}$, \eqref{proof:lem52-item1:neq6} implies that
		\begin{displaymath}
			\begin{aligned}
				H(x^k,y^k,z^k,v^k,c_k) &\leq H(x^{k-1},y^{k-1},z^{k-1},v^{k-1},c_{k-1})	\\ 
				&- (a_k-\tilde{L})\|v^k\|^2	\\
				&\leq H(x^{k-1},y^{k-1},z^{k-1},v^{k-1},c_{k-1}) - a\|v^k\|^2,
			\end{aligned}
		\end{displaymath}
		where $a:=\left( \frac{1}{\overline{t}}-\frac{L_{\nabla r}}{2}-M_1^2M_yL_g'-\frac{M_1^2M_fL_{\nabla g}}{2}-\tilde{L} \right)$. It is obvious that $a>0$ when $\overline{t}<\frac{1}{2\tilde{c}}$ with $\tilde{c}$ defined in \eqref{prop8:eq0}. Then, we get Item (i) immediately.
		
		We next show Item (ii). When $\mathcal{M}$ is an open submanifold of $\mathbb{R}^n$, $\mathrm{T}\mathcal{M}$ is an open submanifold of $\mathbb{R}^n \times \mathbb{R}^n$ \cite[Theorem 3.43]{Boumal:Intro_Opt_Man:2023}. In this case $\mathrm{N}_{(x^k,v^k)}\mathrm{T}\mathcal{M}=\{0_{2n}\}$ and set $\Phi_k=0$ and $s^k=0$ for $k\in\mathbb{N}$. Otherwise, the  local defining function for $\mathcal{M}$ at $x^k$ can be chosen from the set of finite local defining functions $\Phi:=\{\varphi_1,\cdots,\varphi_q:\varphi_i:\mathbb{R}^n\to\mathbb{R}^d,~i=1,\cdots,q\}$ due to the boundedness of $\{x^k:k\in\mathbb{N}\}$. \cite[Theorem 3.43]{Boumal:Intro_Opt_Man:2023} tells us that $\mathrm{T}\mathcal{M}$ is an embedded submanifold of $\mathbb{R}^n\times\mathbb{R}^n$ and $\hat{\Phi}_k:\mathbb{R}^n\times\mathbb{R}^n\to\mathbb{R}^{2d}$ defined at ($x$, $v$) as
		$\hat{\Phi}_k(x,v) = \begin{bmatrix}
			\Phi_k(x)  \\
			\mathrm{D}\Phi_k(x)(v)
		\end{bmatrix}$ is a local defining function for $\mathrm{T}\mathcal{M}$ at ($x^k$, $v^k$), where $\Phi_k\in\Phi$ is a local defining function for $\mathcal{M}$ at $x^k$. One can check from the proof of \cite[Theorem 3.43]{Boumal:Intro_Opt_Man:2023} that 
		\begin{displaymath}
			\mathrm{D}\hat{\Phi}_k(x^k,v^k)
			\begin{pmatrix}
				x\\
				v
			\end{pmatrix} = \begin{bmatrix}
				\mathrm{D}\Phi_k(x^k) & 0\\
				\mathrm{D}\left(\mathrm{D}\Phi_k(\cdot)(v^k)\right)(x^k) & \mathrm{D}\Phi_k(x^k)
			\end{bmatrix}\begin{pmatrix}
				x \\
				v
			\end{pmatrix}.
		\end{displaymath}
		Therefore, 
		\begin{displaymath}
			\begin{aligned}
				\mathrm{N}_{(x^k,v^k)}\mathrm{T}\mathcal{M} &= \{ \mathrm{D}\hat{\Phi}_k(x^k,v^k)^{*}(s): s\in\mathbb{R}^{2d} \} \\
				&= \left\{ \begin{bmatrix}
					\mathrm{D}\Phi_k(x^k)^{*}(s_1) + \mathrm{D}\left( \mathrm{D}\Phi_k(\cdot)(v^k) \right)(x^k)^{*}(s_2)\\
					\mathrm{D}\Phi_k(x^k)^{*}(s_2)
				\end{bmatrix} : s_1,~s_2\in\mathbb{R}^d \right\}.
			\end{aligned}
		\end{displaymath}
		Let $\{s^k\in\mathbb{R}^d:k\in\mathbb{N}\}$ be the sequence in Lemma~\ref{lemma:Chapter6:lem62}. We set $u^k=\times_{i=1}^{5}u_i^k$ with
		\begin{displaymath}
			\begin{aligned}
				u_1^k &:= (\partial h_1 + \nabla r)(x^k+v^k) - z^k + c_k^2 \nabla g(x^k+v^k)\left(\langle x^k + v^k,y^k\rangle - f^*(y^k)\right) \\&+ c_k^2g(x^k+v^k)y^k - 2c_ky^k  + \mathrm{D} \Phi(x^k)^{*}(s^k) + \mathrm{D}\left(\mathrm{D}\Phi_k(\cdot)(v^k) \right)(x^k)^{*}(s^k), \\
				u_2^k &:=\left(c_k^2g(x^k+v^k)-2c_k\right)\left(x^k+v^k-\partial f^*(y^k)\right),	\\
				u_3^k &:= -x^k-v^k+\partial h_2^{*}(z^k),	\\
				u_4^k &:=(\partial h_1 + \nabla r)(x^k+v^k) - z^k + c_k^2 \nabla g(x^k+v^k)\left(\langle x^k + v^k,y^k\rangle - f^*(y^k)\right) \\&+ c_k^2g(x^k+v^k)y^k - 2c_ky^k + \mathrm{D} \Phi(x^k)^{*}(s^k) + 2 \tilde{L}v^k,	\\
				u_5^k &:= \left(2c_kg(x^k+v^k)-2\right)\left(\langle x^k + v^k,y^k\rangle - f^*(y^k)\right).
			\end{aligned}
		\end{displaymath}
		Since $\mathrm{T}\mathcal{M}$ is an embedded submanifold of $\mathbb{R}^n\times\mathbb{R}^n$, $\mathrm{T}\mathcal{M}$ is regular. It is not hard to verify that \cite[Proposition 2.2]{Zhang-Li:SIOPT:2022} $\widehat{\partial}(\varphi_1\varphi_2)(\overline{x})=\varphi_2(\overline{x})\nabla\varphi_1(\overline{x})+\varphi_1(\overline{x})\widehat{\partial}\varphi_2(\overline{x})$ when $\varphi_1:\mathbb{R}^n\to\mathbb{R}$ is continuously differentiable at $\overline{x}$ with $\varphi_1(\overline{x})>0$ and $\varphi_2:\mathbb{R}^n\to(-\infty,+\infty]$ is continuous at $\overline{x}$ relative to $\mathrm{dom}(\varphi_2)$, and here we suppose $\varphi_1\varphi_2(x)=+\infty$ when $x \notin \mathrm{dom}(\varphi_2)$. It is also obtained from \cite[Proposition 2.1]{Attouch-Bolte-Redont-Soubeyran:MOR:2010} that $\widehat{\partial}\varphi(x,y)=\widehat{\partial}\varphi_1(x)\times\widehat{\partial}\varphi_2(y)+\nabla\varphi_3(x,y)$, for $\varphi:\mathbb{R}^n\times\mathbb{R}^m\to(-\infty,+\infty]$ defined at $(x,y)$ as $\varphi(x,y)=\varphi_1(x)+\varphi_2(y)+\varphi_3(x,y)$ with $\varphi_1:\mathbb{R}^n\to(-\infty,+\infty]$, $\varphi_2:\mathbb{R}^m\to(-\infty,+\infty]$ being proper lower semicontinuous and $\varphi_3:\mathbb{R}^n\times\mathbb{R}^m\to\mathbb{R}$ being continuously differentiable. These subdifferential calculus together with \eqref{prop8:eq7}, Assumption~\ref{ass1} and Assumption~\ref{assumption:Chapter6} yield that $u^k\subseteq \widehat{\partial} H(x^k,\tilde{y}^k,z^k,v^k,c_k)$. We finally show $\mathrm{dist}(0_{4n+1},u^k)\leq b\|v^k\|$ for some $b>0$ for all $k\in\mathbb{N}$. It suffices to show that there exist $b_i>0$ such that $\mathrm{dist}(0_n,u_i^k)\leq b_i\|v^k\|$ for $k\in\mathbb{N}$ and $i=1,\cdots,5$. Since $\Phi_k\in\Phi$ and $\Phi$ has finite elements, we have $\|\mathrm{D}\left(\mathrm{D}\Phi_k(\cdot)v^k\right)(x^k)^{*}(s^k) \|\leq \hat{b}\|v^k\|$ for some $\hat{b}>0$ from the fact that $\Phi_k$ is infinitely differentiable and the fact that $\{(x^k,s^k):k\in\mathbb{N}\}$ is bounded due to Theorem~\ref{theorem:FunctionConvergence} and Lemma~\ref{lemma:Chapter6:lem62}. The convexity of $f$ and $y^k\in\partial f(x^k)$ imply that $\langle x^k+v^k,y^k\rangle-f^*(y^k)=f(x^k)+\langle v^k,y^k\rangle$ and $\{y^k:k\in\mathbb{N}\}$ is bounded due to Theorem~\ref{theorem:FunctionConvergence} (iii). This together with Lemma~\ref{lemma:Chapter6:lem62} and Lemma~\ref{lem1} (i), as well as the local Lipschitz differentiability of $\nabla r$ and $\nabla g$, yields that $\mathrm{dist}(0,u_1^k)\leq b_1\|v^k\|$ for some $b_1>0$ and for all $k\in\mathbb{N}$ due to $\lim\limits_{k\to\infty}\|v^k\|=0$. Similarly, there exists $b_4>0$ such that $\mathrm{dist}(0_n,u_4^k)\leq b_4\|v^k\|$ for $k\in\mathbb{N}$. Invoking $y^k \in \partial f(x^k)$ and $z^k\in\partial h_2(x^k)$, it follows from Lemma~\ref{lem1} (i) that $\mathrm{dist}(0_n,u_2^k)$, $\mathrm{dist}(0_n,u_3^k)\leq b_2\|v^k\|$, for some $b_2>0$ and for all $k\in\mathbb{N}$. Thanks to $g\geq m_g$ on $\mathcal{M}$, the locally Lipschitz continuity of $g$, and the fact that $\{y^k\in\partial f(x^k):k\in\mathbb{N}\}$ is bounded due to the convexity of $f$ and the boundedness of $\{x^k:k\in\mathbb{N}\}$, we deduce that $\mathrm{dist}(0,u_5^k)\leq b_5\|v^k\|$ for some $b_5>0$ and for all $k\in\mathbb{N}$. We complete the proof immediately.
	\end{proof}

	Now, we are ready to establish the main result of this section.
	\begin{theorem} \label{theorem:sequentialConvergence}
		Suppose Assumption~\ref{assumption:section4:compactSetX} and Assumption~\ref{assumption:Chapter6} hold. Let $\{x^k:k\in\mathbb{N}\}$ be generated by MPGSA and $\Omega$ be the set of all accumulation points of $\{x^k:k\in\mathbb{N}\}$. If $\overline{t}<\min\left\{\frac{1}{2c}, \frac{1}{2\tilde{c}}, \frac{m_g}{2L_g'(L_1+M_2)}\right\}$ with $c$, $\tilde{c}$, $m_g$, $L_g'$, $L_1$ and $M_2$ being defined as the same as those in Proposition~\ref{Case1:Prop1} and $H$ being defined in \eqref{eq:definition:H} satisfies the KL property at every point of $\Omega \times \mathrm{dom}(\partial f^*) \times \mathrm{dom}(\partial h_2^*) \times \mathbb{R}^n \times \mathbb{R}$, then $\sum\limits_{k=1}^{+\infty}\|v^k\|<+\infty$ and $\{x^k:k\in\mathbb{N}\}$ converges to a critical point $x^*$ of problem \eqref{problemn1}. 
	\end{theorem}
	\begin{proof}
		According to Proposition~\ref{Case1:Prop1} (i), we have $\lim\limits_{k\to\infty} H(u^k)=H_{*}$ for some $H_{*}\in\mathbb{R}$ from Lemma~\ref{lemma:Q_property} (i) and the fact that problem \eqref{problemn1} has a solution, where $u^k:=(x^k, y^k, z^k, v^k, c_k)$. Theorem~\ref{theorem:FunctionConvergence} (ii)-(iii) and Assumption~\ref{ass1} indicate that $\{u^k:k\in\mathbb{N}\}$ is bounded. This implies that the set of all accumulation points of $\{u^k:k\in\mathbb{N}\}$, denoted by $U$, is compact. Then, it is not hard to deduce that $H(u^{*})=H_{*}$ for all $u^{*}\in U$, thanks to the continuity of $h$, $r$, $g$ and $f$ as well as the outer semicontinuous of $\partial h_2$ and $\partial f$ \cite[Proposition 8.7]{Rockafellar-Wets:VariationalAnalysis:2009}. Thus, invoking Lemma~\ref{lemma:Uniformized_KL_property}, $H$ satisfies the uniformized KL property on the compact set $U$. As the remaining proof is rather standard, we omit it here and refer the readers to \cite{Attouch:MP:2009,Attouch-Bolte-Redont-Soubeyran:MOR:2010,Attouch:MP:2013}.
	\end{proof}
	
	Next, we show that when $f$ is locally differentiable on $\mathbb{R}^n$, we can also derive the sequential convergence of MPGSA without Assumption~\ref{assumption:Chapter6}. In this case, $\frac{f}{g}$ is locally Lipschitz differentiable on an open set containing $\mathcal{M}$. Thus, by setting $\widetilde{H}:\mathbb{R}^n\times\mathbb{R}^n\times\mathbb{R}^n\to(-\infty,+\infty]$ as
	\begin{equation}  \label{eq:ChapterEnd:star}
		\widetilde{H}(x,z,v) := \left( h_1+r+\frac{f}{g} \right)(x+v)+h_2^{*}(z)-\langle x+v,z \rangle + \iota_{\mathrm{T}\mathcal{M}}(x,v) + \tilde{L}\|v\|^2
	\end{equation}
	and following a similar line of arguments to the analysis in Proposition~\ref{Case1:Prop1} and Theorem~\ref{theorem:sequentialConvergence}, we can obtain the following theorem. Due to the limitation of space, we omit the detail of the proof.
	\begin{theorem}\label{111111}
		Suppose Assumption~\ref{assumption:section4:compactSetX} holds and $f$ is locally Lipschitz differentiable on $\mathbb{R}^n$. Let $\{x^k:k\in\mathbb{N}\}$ be generated by MPGSA.  Let $\{x^k:k\in\mathbb{N}\}$ be generated by MPGSA and $\Omega$ be the set of all accumulation points of $\{x^k:k\in\mathbb{N}\}$. If $\overline{t}<\min\left\{\frac{1}{2c}, \frac{1}{2\tilde{c}}, \frac{m_g}{2L_g'(L_1+M_2)}\right\}$ with $c$, $\tilde{c}$, $m_g$, $L_g'$, $L_1$ and $M_2$ being defined as the same as those in Proposition~\ref{Case1:Prop1} and $\widetilde{H}$ defined in \eqref{eq:ChapterEnd:star} satisfies the KL property at every point of $\Omega \times\mathrm{dom}(\partial h_2^*)\times\mathbb{R}^n$, then $\{x^k:k\in\mathbb{N}\}$ converges to a critical point problem \eqref{problemn1}.
	\end{theorem}
	
	To close this section, we demonstrate the sequential convergence of MPGSA when applied to the application examples introduced in Section~\ref{sec:intro}. Clearly, the numerator function $f$ in each example not only satisfies Assumption~\ref{ass1} but is also locally Lipschitz differentiable on $\mathbb{R}^n$. Therefore, for simplicity, we will base our argument on Theorem~\ref{111111} rather than Theorem~\ref{theorem:sequentialConvergence}. Recall that in these examples, the underlying manifold $\mathcal{M}$ is the Stiefel manifold $\mathcal{S}_{n,p}$, whose tangent bundle $\mathrm{T}\mathcal{S}_{n,p}$ is given by
	$
	\{(X,V)\in\mathbb{R}^{n\times p}\times\mathbb{R}^{n\times p}:X^TX=I_p,\ X^TV+V^TX=0\}
	$.
	Thanks to the compactness of $\mathcal{S}_{n,p}$, Assumption~\ref{assumption:section4:compactSetX} is automatically satisfied by MPGSA with any initial point. Furthermore, $\mathrm{T}\mathcal{S}_{n,p}$ is a semialgebraic set, and the functions $f$, $g$, $r$, $h_1$ and $h_2^*$ are all semialgebraic. Since the sum or quotient of two semialgebraic functions remains semialgebraic, and the indicator function of a semialgebraic set is also semialgebraic, it follows that the auxiliary function $\widetilde{H}$ is semialgebraic in each example, and hence is a KL function. In light of the above reasoning and by Theorem~\ref{111111}, we immediately obtain the sequential convergence of MPGSA for all three application examples from Section~\ref{sec:intro}.

\section{Numerical experiments.}
	In this section, we test the performance of our proposed algorithms, namely, MPGSA and EMPGSA, by applying them to the sparse FDA problem \eqref{model:SFDA} with two special choices of regularization functions. The first one is the $\ell_{1}$ regularized problem
	\begin{equation} 	\label{model:SFDA_l21}
		\min\limits_{X \in \mathcal{S}_{n,p}} \left\{ \lambda \|X\|_{1}-\frac{\mathrm{tr}(X^TA_b X)}{\mathrm{tr}(X^TA_w X)} \right\},
	\end{equation}
	where the $\ell_{1}$ norm $\|\cdot\|_{1}$ is defined as $\|X\|_{1}=\sum_{i=1}^{n} \sum_{j=1}^{p}|x_{ij}|$. Another example is the partial $\ell_1$ function regularized problem
	\begin{equation} 	\label{model:PartialL1functionRegularizedproblem}
		\min\limits_{X \in \mathcal{S}_{n,p}} \left\{ \lambda \|X\|_{1} - \lambda\|X\|_{(K)}-\frac{\mathrm{tr}(X^TA_b X)}{\mathrm{tr}(X^TA_w X)} \right\},
	\end{equation}
	where $1\leq K\leq n$ is an integer and for $X\in\mathbb{R}^{n\times p}$, $\|X\|_{(K)}$ denotes the sum of the $K$ largest absolute values of entries in $X$. We compare MPGSA with the semismooth Newton based augmented Lagrangian method ALMSN \cite{Zhou-Bao-Ding-Zhu:MP:2023} for solving problem \eqref{model:SFDA_l21}. It is compared in our experiments that two versions of ALMSN, namely, ALMSN1 and ALMSN2, which globalize the semismooth Newton method involved in different ways. Since there is no existing method available for directly solving problem \eqref{model:PartialL1functionRegularizedproblem} to the best of our knowledge, we only test the efficiency of MPGSA and EMPGSA for solving this example. All the numerical tests are conducted on a HP desktop with a 3.0 GHz Intel core i5 and 16GB of RAM, equipped with MATLAB R2022a.
	
	We first describe the implementation details of the aforementioned algorithms. For MPGSA and EMPGSA, we set $\gamma=0.5$ and choose $t_k$ to be the Barzilai-Borwein (BB) stepsize \cite{Barzilai-Borwein:IMAJNumAnal:1988} as follows
	\begin{align*}
		t_k = \left\{
		\begin{array}{ll} 
			\max\left\{\underline{t}, \min\left\{\overline{t},\dfrac{\|\Delta x^k\|^2}{|\langle\Delta x^k, \Delta l^k \rangle|}\right\}\right\} &,  \langle\Delta x^k, \Delta l^k \rangle \neq 0;\\
			\overline{t}  &, \langle\Delta x^k, \Delta l^k \rangle = 0,
		\end{array}
		\right. 
	\end{align*}
	where $\Delta x^k = x^k - x^{k-1}$, $\Delta l^k = \nabla l_k(x^k) - \nabla l_k(x^{k-1})$, $l_k(x) = r(x) + \frac{f(x^k)}{g^2(x^k)}g(x)$. Also, we choose $\eta=10^{-8}$ for EMPGSA. The MPGSA is terminated when the number of iterations exceeds 10000 or $\|v^k\|_F^2/t_k^2 < 10^{-8}np$ with $v^k$ being defined by \eqref{eqAlgV}, while the EMPGSA is terminated once the number of iterations hits 10000 or $\|v^{k,\hat{i}}\|_F^2/t_k^2 < 10^{-8}np$ with $v^{k,\hat{i}}$ being defined by \eqref{eq:subproblem_MPGSA_d} and Step 3 (c) in EMPGSA. Moreover, the subproblem \eqref{eqAlgV} in MPGSA and the subproblem \eqref{eq:subproblem_MPGSA_d} in EMPGSA are solved by the semismooth Newton method developed in \cite[Section 4.2]{Chen-Ma-Anthony-Zhang:SIOPT:2020}. For the parameters and termination conditions of ALMSN1 and ALMSN2, we simply adopt the suggested setting in \cite[Section 5.2]{Zhou-Bao-Ding-Zhu:MP:2023} but set the termination threshold to be $10^{-8}np$. Finally, we remark that for all the compared algorithms, the polar decomposition is used as the retraction mapping.
	
	In the first experiment, we simulate the matrices $A_b\in\mathbb{R}^{n\times n}$ and $A_w\in\mathbb{R}^{n\times n}$ of problem \eqref{model:SFDA_l21} and \eqref{model:PartialL1functionRegularizedproblem} following a similar setting to \cite[Section 5.1]{Tan-Wang-Liu-Zhang:JRSSSB:2018}, which uses sparse FDA for multi-class classification. Specifically, we consider a four-class classification example as follows. In the $i$-th class, we randomly generate 250 samples with $n= 500$ features following a Gaussian distribution $\mu^i\in\mathbb{R}^n$ and covariance $\Sigma\in \mathbb{R}^{n\times n}$. For $i\in\left\{ 1,2,3,4 \right\}$, we set $\mu_j^i=(i-1)/3$ for $j\in\left\{2,4,...,40\right\}$ and $\mu_j^i =0$ otherwise. Meanwhile, $\Sigma$ is chosen as a block diagonal matrix with five blocks, each of which is in the dimension of ($n/5$) $\times$ ($n/5$). The ($j$, $j^\prime$)th entry of each block is set as $0.8^{|j-j^\prime|}$. With these four classes of samples, we further let $A_b$ and $A_w$ be the associated between-class covariance and within-class covariance matrices. Next, the MPGSA, ALMSN1 and ALMSN2 are applied to solving problem \eqref{model:SFDA_l21}, while the MPGSA and EMPGSA are applied to solving problem \eqref{model:PartialL1functionRegularizedproblem}. We set $p=3$, fix the regularization parameter $\lambda= 0.21$ for problem \eqref{model:SFDA_l21} and choose $K=50,\ \lambda=0.22$ for problem \eqref{model:PartialL1functionRegularizedproblem}. All the algorithms are started from the same randomly chosen initial point on $\mathcal{S}_{n,p}$. Moreover, we evaluate the quality of the solutions obtained by each of compared algorithms through studying its associated classification accuracy as described below. First, we use the aforementioned samples as the training dataset and randomly generate 500 more samples for each class to form the testing dataset. Second, the samples in the training dataset are projected into the subspace spanned by the columns of the solution. Then the projected points are employed to train a support vector machine (SVM) with a linear kernel which is implemented using the LIBSVM toolbox \cite{chang2011libsvm}. Finally, we project the samples of the testing dataset into the subspace and apply the trained SVM to classifying these projected datapoints. The classification accuracy is measured by the number of the correctly classified samples over the total testing samples.

	\begin{table}[htbp]
			\caption{Computational results for the first experiment.}
		\centering
		\label{table:SFDA_table1}
			\begin{tabular}{|c|cccc|}
				\hline
				Alg. & Time & Accuracy & Sparsity & Iteration\\ \hline      
				MPGSA for \eqref{model:SFDA_l21} & 0.21 & 84.34\% & 0.901 & 141 \\
				ALMSN1 for \eqref{model:SFDA_l21} & 0.71 & 84.32\% & 0.903 & 53 \\ 
				ALMSN2 for \eqref{model:SFDA_l21} & 1.54 & 84.32\% & 0.903 & 61 \\                            
				MPGSA for \eqref{model:PartialL1functionRegularizedproblem} & 0.27 & 85.14\% & 0.905 & 228 \\                
				EMPGSA for \eqref{model:PartialL1functionRegularizedproblem} & 0.31 & 85.13\% & 0.905 & 229 \\
				\hline
		\end{tabular}%
	\end{table}

	Table \ref{table:SFDA_table1} summarizes the computational results of the compared algorithms averaged over 50 rounds of random tests. In detail, we report the CPU time for solving problem \eqref{model:SFDA_l21} or \eqref{model:PartialL1functionRegularizedproblem} (Time), the classification accuracy in percentage (Accuracy), the number of nonzero entries in the obtained solutions over $np=1500$ (Sparsity), the number of iterations (Iteration) averaged over 50 rounds of random tests. One can observe that the MPGSA and EMPGSA for solving \eqref{model:SFDA_l21} and \eqref{model:PartialL1functionRegularizedproblem} outperform the other algorithms in terms of CPU time. Moreover, the solutions found by the MPGSA and EMPGSA for solving \eqref{model:PartialL1functionRegularizedproblem} have a bit higher sparsity, while the classification accuracy of them are better than those of the other methods or solving \eqref{model:SFDA_l21}. We also observe that the classification accuracy of the MPGSA for solving \eqref{model:PartialL1functionRegularizedproblem} is almost the same as the EMPGSA. It is probably because in this experiment the MPGSA converges to a critical point where $\|\cdot\|_{(K)}$ is continuously differentiable, and thus this point is also a lifted B-stationary point as proved in Proposition \ref{proposition_28_definition}(\romannumeral6).
	
	In order to further compare MPGSA and EMPGSA, we conduct the second experiment by applying them to constructed instances of problem \eqref{model:PartialL1functionRegularizedproblem} with $p=1$. Given positive integers $n$, $K$ and a positive number $\lambda$, we generate a symmetric positive semidefinite matrix $A_b$, a symmetric positive definite matrix $A_w$, and a critical but not lifted B-stationary point $\overline{x}$ of problem \eqref{model:PartialL1functionRegularizedproblem} as follows. Specifically, we first produce a vector $\tilde{\xi}\in\mathbb{R}^K$ with components randomly drawn from the standard Gaussian distribution, and thus obtain a vector $\xi\in\mathbb{R}^K$ by sorting the components of $\xi$ such that $|\xi_1|\geq|\xi_2|\geq\cdots\geq|\xi_K|$. Then a vector $\overline{x}\in\mathbb{R}^n$ is generated by setting $\overline{x}_i=\xi_i+\mathrm{sign}(\xi_i)$ for $i=1,2,\cdots,K$, $\overline{x}_{K+1}=\overline{x}_{K+2}=\overline{x}_K$, and $\overline{x}_i=0$ otherwise, where $\mathrm{sign}(\cdot)$ denotes the classic sign function of a scalar. Clearly, it holds that $|\overline{x}_1|\geq|\overline{x}_2|\geq\cdots\geq|\overline{x}_K|=|\overline{x}_{K+1}|=|\overline{x}_{K+2}|>|\overline{x}_{K+3}|=\cdots=|\overline{x}_n|=0$ and hence $\|x\|_{(K)}$ is not differentiable at $\overline{x}$. Next, we generate a vector $\tilde{\zeta}\in\mathbb{R}^n$ with elements randomly chosen from the uniform distribution on $[1,10]$, and obtain the matrix $A_w$ by letting $A_w =\mathrm{Diag}(\tilde{\zeta})$, where $\mathrm{Diag}(\tilde{\zeta})$ is the diagonal matrix with the components of $\tilde{\zeta}$ on the diagonal. Let $\theta:= \overline{x}^TA_w \overline{x}$. Finally, we compute the nonnegative vector $\zeta\in\mathbb{R}^n$ and $\gamma\in\mathbb{R}$ such that for some $w^1\in\lambda \partial (\|\cdot\|_1)(\overline{x})$ and $w^2\in\lambda \partial (\|\cdot\|_{(K)})(\overline{x})$, it holds that
	\begin{displaymath}
		0_n = w^1 - w^2 + U \zeta + \gamma \overline{x},
	\end{displaymath}
	where $U := A_w \overline{x} \overline{x}^T \mathrm{Diag}(\overline{x})/\theta^2 - 2\mathrm{Diag}(\overline{x})/\theta$. In our implementation, we solve the following convex quadratic program to obtain the desirable $\zeta$ and $\gamma$
	\begin{displaymath}
		\min \limits_{\zeta\in\mathbb{R}^n,~\gamma\in\mathbb{R}} \left\{ \|U\zeta + \gamma\overline{x}+w^1-w^2\|_2^2:\zeta\geq0 \right\}.
	\end{displaymath}
	
	In view of Definition~\ref{definition:lifted_sta_And_critical} (i) and (ii), it is not hard to verify that $\overline{x}$ is a critical but not lifted B-stationary point of problem \eqref{model:PartialL1functionRegularizedproblem} with $A_b =\mathrm{Diag}(\zeta)$, the aforementioned $A_w$ and other parameters of the problem.  Moreover, one can easily see that the global optimum of the resulting problem is equal to $\min\{- \zeta_i/ \Tilde{\zeta_i}: i=1,2,...,n \}$. In this experiment, we fix $n=100,\ K=3$ and $\lambda=1.2$. We first construct an instance of problem \eqref{model:PartialL1functionRegularizedproblem} and a critical but not B-stationary point $\overline{x}$ of it as described above. Then we apply MPGSA and EMPGSA to this instance and start them from the initial point $x^0 = \overline{x}+0.01\beta\mathbf{1}_n$, where $\beta\in\mathbb{R}$ are drawn randomly from the uniform distribution on $[-1,1]$.
	\par Table~\ref{table:SFDA_table2} reports the computational results averaged over 50 random tests. We report the objective value (Objective), CPU time in seconds (Time), the number of nonzero components in the obtained solutions over $n$ (Sparsity) and the number of iterations (Iteration) averaged over 50 random tests. Besides, we also record the number of tests where the difference of the objective values found by the algorithms and the global optimum is less than $10^{-6}$ and report the ratio of them and the total number of tests in percentage (Optimum). One can see that the EMPGSA attains the global optimum in all the tests, which significantly outperforms the MPGSA and thus leads to much lower averaged objective values. This is not surprising because EMPGSA is proved to converges to a lifted B-stationary point, which is generally stronger than a critical point. Although each iteration of EMPGSA involves solving more subproblems than MPGSA, it requires significantly fewer iterations, resulting in a lower total CPU time.
	
	\begin{table}[h] 
		\caption{Computational results for the second experiment.}
		\centering
		\label{table:SFDA_table2}          
			\begin{tabular}{|c|ccccc|}
				\hline
				Alg. & \multicolumn{1}{l}{Objective} & \multicolumn{1}{l}{Optimum} & \multicolumn{1}{l}{Time} & \multicolumn{1}{l}{Sparsity} & \multicolumn{1}{l|}{Iteration}\\ \hline 
				MPGSA for \eqref{model:PartialL1functionRegularizedproblem} & -2.00 & 36\% & 0.0073 & 0.97 & 331 \\ 
				EMPGSA for \eqref{model:PartialL1functionRegularizedproblem} & -2.91 & 100\% & 0.0026 & 0.97 & 11 \\ \hline
			\end{tabular}
	\end{table}
	
\section{Conclusion.}
	In this paper, we consider a class of structured nonsmooth DC and fractional minimization over Riemannian submanifolds as described in problem \eqref{problemn1}. This problem is general and when $h_2=f=0$ and $g$ is a positive constant, it reduces to the model problem which is considered by many existing works on nonsmooth manifold optimization. We first characterize the optimality of problem \eqref{problemn1} by various notions of stationarity and examine their relationships. Then we propose the MPGSA and show that under mild conditions every accumulation point of the solution sequence generated by it is a critical point of problem \eqref{problemn1}. Under certain suitable conditions, the convergence of the entire solution sequence generated by the MPGSA is further established by assuming the KL property of some auxiliary function. Moreover, in the case where $h_2$ is the maximum of finite continuously differentiable convex functions, we also propose the EMPGSA and establish its subsequential convergence to a lifted B-stationary point of problem \eqref{problemn1}. Finally, we illustrate the proposed methods by two preliminary numerical experiments.

\begin{acknowledgements}
	We would like to thank Prof. Man-Cho So for helpful discussions.
\end{acknowledgements}

\bibliography{sample}
\end{document}